\documentclass[12pt,a4paper]{amsart}
\usepackage[hmargin=32mm, vmargin=27mm, includefoot, twoside]{geometry}
\usepackage[backref=page,colorlinks,bookmarksopen=true]{hyperref}
\usepackage[latin1]{inputenc}  
\usepackage[T1]{fontenc}       
\usepackage[english]{babel}
\usepackage[dvips]{graphics}
\usepackage{mathrsfs}
\usepackage{amssymb}
\usepackage{xspace}
\usepackage[alphabetic]{amsrefs}
\usepackage{enumerate}

\numberwithin{equation}{section}

\newtheorem{thm}{Theorem}[section]

\newtheorem{lem}[thm]{Lemma}
\newtheorem{cor}[thm]{Corollary}
\newtheorem{prop}[thm]{Proposition}

\newtheorem{propapp}{Proposition}

\newtheorem{lemapp}[propapp]{Lemma}

\theoremstyle{definition}

\theoremstyle{remark}
\newtheorem{rem}[thm]{Remark}
\newtheorem{ex}[thm]{Example}

\newcommand{\tref}[1]{Theorem~\ref{#1}}
\newcommand{\cref}[1]{Corollary~\ref{#1}}
\newcommand{\pref}[1]{Proposition~\ref{#1}}

\newcommand{\lref}[1]{Lemma~\ref{#1}}

\newcommand{\diam}{\mathrm{diam}}

\newcommand{\rad}{\mathrm{rad}}

\newcommand{\id}{\mathrm{Id}}

\newcommand{\RR}{\mathbf{R}}

\newcommand{\centra}{\mathscr{Z}}
\newcommand{\Isom}{\mathrm{Is}}
\newcommand{\la}{\langle}
\newcommand{\ra}{\rangle}

\newcommand{\cat}{{\upshape CAT(0)}\xspace }
\newcommand{\catone}{{\upshape CAT(1)}\xspace }
\newcommand{\bd}{\partial}
\newcommand{\bdfine}{\partial_{\mathrm{fine}}}
\newcommand{\vareps}{\varepsilon}

\renewcommand{\theenumi}{\roman{enumi}}

\begin{document}

\title{At infinity of finite-dimensional \cat spaces}
\author[Pierre-Emmanuel Caprace]{Pierre-Emmanuel Caprace*}
\address{UCL -- Math, Chemin du Cyclotron 2, 1348 Louvain-la-Neuve, Belgium}
\email{pe.caprace@uclouvain.be}
\thanks{*F.N.R.S. Research Associate}

\author{Alexander Lytchak$^\dagger$}
\address{Mathematisches Institut, Universit\"at  Bonn, Beringstrasse 1, D-53115 Bonn}
\email{lytchak@math.uni-bonn.de}
\thanks{$^\dagger$Supported in part by SFB 611 and MPI f\"ur Mathematik in Bonn}

\date{October 9, 2008}

\begin{abstract}
We show that any filtering family of closed convex subsets of a finite-dimensional \cat space $X$ has a
non-empty intersection in the visual bordification $ \overline{X} = X \cup \bd X$. Using this fact, several
results known for proper \cat spaces may be extended to finite-dimensional spaces, including the existence of
canonical fixed points at infinity for parabolic isometries, algebraic and geometric restrictions on amenable
group actions, and geometric superrigidity for non-elementary actions of irreducible uniform lattices in
products of locally compact groups.
\end{abstract}


\maketitle

\section{Introduction}


Several families of finite-dimensional \cat spaces naturally include specimens which are not locally compact;
\emph{e.g.} buildings of finite rank (Euclidean or not), finite-dimensional \cat cube complexes, or asymptotic
cones of Hadamard manifolds or of \cat groups.

A major difficulty one encounters when dealing with non-proper spaces is that the visual boundary may have a
very pathological behaviour. For example, an {unbounded} \cat space may well have an \emph{empty} visual
boundary. The purpose of this paper is to show that for finite-dimensional spaces, the visual boundary
nevertheless enjoys similarly nice properties as in the case of proper spaces.

Following B.~Kleiner \cite{Kleiner}, we define the \textbf{(geometric) dimension} of a \cat space $X$ to be the
supremum over all compact subsets $K \subset X$ of the topological dimension of $K$. We refer to
\emph{loc.~cit.} for more details and several characterizations of this notion. A $0$-dimensional \cat space is
reduced to a singleton, while $1$-dimensional \cat spaces coincide with $\RR$-trees. We emphasize that the
notion of geometric dimension is \emph{local}. It turns out that, for our purposes, it will be sufficient to
demand that the spaces have finite dimension \emph{at large scale}. In order to define this condition precisely,
we shall say that a \cat space $X$ has \textbf{telescopic dimension} $\leq n$ if every asymptotic cone
$\lim_\omega (\vareps_n X, \star_n)$ has geometric dimension~$\leq n$. A space has telescopic dimension~$0$ if
and only if it is bounded. It has telescopic dimension~$\leq 1$ if and only if it is Gromov hyperbolic. A \cat
space of finite geometric dimension has finite telescopic dimension. We refer to \S\ref{sec:dimension} below for
more details and some examples.

\begin{thm}\label{filtering}
Let $X$ be a complete \cat space of finite telescopic dimension and $\{X_\alpha\}_{\alpha \in A}$ be a filtering
family of closed convex subspaces. Then either the intersection $\bigcap_{\alpha \in A} X_{\alpha}$ is
non-empty, or the intersection of the visual boundaries $\bigcap_{\alpha \in A} \bd X_{\alpha}$ is a non-empty
subset of $\bd X$ of intrinsic radius at most $\pi /2$.
\end{thm}

Recall that a family $\mathcal{F}$ of subsets of a given set is called \textbf{filtering} if for all $E, F \in
\mathcal{F}$ there exists $D \in \mathcal{F}$ such that $D \subseteq E \cap F$. In particular the preceding
applies to nested families of closed convex subsets, and provides a criterion ensuring that the visual boundary
$\bd X$ is non-empty. In the course of the proof, we shall establish a result similar to Theorem~\ref{filtering}
for finite-dimensional $\mathrm{CAT}(1)$ spaces (see \pref{prop:cat1:filtering} below). We remark however that
Theorem~\ref{filtering} fails for complete \cat spaces with finite-dimensional Tits boundary, see
Example~\ref{ex:Hilbert} below.

\begin{rem}
Theorem~\ref{filtering} may be reformulated using the topology $\mathscr{T}_c$ introduced by Nicolas
Monod~\cite[\S3.7]{Monod_superrigid} on the set $\overline{X} = X \cup \bd X$. It is defined as the coarsest
topology such that for any convex subset $Y \subseteq X$, the (usual) closure $\overline Y$ in $\overline X$ is
$\mathscr{T}_c$-closed. It is known that any bounded closed subset of $X$ is $\mathscr{T}_c$-quasi-compact (see
\cite[Theorem~14]{Monod_superrigid}) and that, if $X$ is Gromov hyperbolic, then $\overline X$ is
$\mathscr{T}_c$-quasi-compact (see Proposition~23 in \emph{loc.~cit.}). However, if $X$ is infinite-dimensional
then $\overline X$ is generally not $\mathscr{T}_c$-quasi-compact. Theorem~\ref{filtering} just means that,
\emph{given a  complete \cat space of finite telescopic dimension, the set  $\overline X$ is quasi-compact for
the topology $\mathscr{T}_c$}. This compactness property is thus shared by proper \cat spaces, Gromov hyperbolic
\cat spaces and finite-dimensional \cat spaces.
\end{rem}

A key idea in the proof of Theorem~\ref{filtering} is to obtain
points at infinity by applying (a very special case of) a result of
A.~Karlsson and G.~Margulis \cite{KarlssonMargulis} to the gradient
flow of a convex function that is associated in a canonical way to
the given filtering family. This strategy requires to show that the
velocity of escape of the gradient flow in question is strictly
positive. This is where the assumption on the telescopic dimension
of the ambient space is used; the main point in estimating that
velocity is the following natural generalisation to non-positively
curved spaces of H.~Jung's classical theorem~\cite{Jung}. Another
closely related generalisation was established in
\cite{LangSchroeder}.

\begin{thm} \label{raddiam}
Let $X$ be a  \cat space and $n$ be a positive integer.

Then $X$ has geometric dimension $\leq n $ if and only if for each subset $Y$ of $X$ of finite diameter we have
$\displaystyle \rad_X (Y) \leq \sqrt {\frac {n} {2(n+1)}} \diam(Y)$.

Similarly $X$ has telescopic dimension $\leq n$ if and only if for any $\delta >0$ there exists some constant $D
> 0$ such that for any bounded subset $Y \subset X$ of diameter~$>D$, we have $\displaystyle \rad_X (Y) \leq
\bigg(\delta +\sqrt {\frac {n} {2(n+1)}} \bigg) \diam(Y)$.
\end{thm}

Recall that the \textbf{circumradius} $\rad_X(Y)$ of a subset $Y \subseteq X$ is defined as the infimum of all
positive real numbers $r$ such that $Y$ is contained in some closed ball of radius $r$ of $X$.

\begin{rem} In the case of an $n$-dimensional regular Euclidean simplex one has equality in the theorem above.
For a short discussion of the case of equality as well as analogous statements in other curvature bounds we
refer to Section~\ref{secjung}.
\end{rem}

\medskip%
It turns out that Theorem~\ref{filtering} provides a key property that allows one to extend to
finite-dimensional \cat spaces several results which are known to hold for proper spaces. We now proceed to
describe a few of these applications.

\subsection*{Parabolic isometries}

A first elementary consequence of Theorem~\ref{filtering} is the existence of canonical fixed points at infinity
for parabolic isometries. This extends the results obtained in \cite[Theorem~1.1]{Nagano} and
\cite[Corollary~2.3]{CapraceMonod} in the locally compact setting.

\begin{cor}\label{FixedPointParabolic}
Let $g$ be a parabolic isometry of a \cat space $X$ of finite telescopic dimension. Then the centraliser
$\centra_{\Isom(X)}(g)$ possesses a canonical fixed point in $\bd X$.
\end{cor}

\subsection*{Amenable group actions}

The next application provides obstructions to isometric actions of amenable groups; in the locally compact case
the corresponding statement is due to S.~Adams and W.~Ballmann \cite{Adams-Ballmann}, and generalizes earlier
results by M.~Burger and V.~Schroeder~\cite{BurgerSchroeder}.

\begin{thm} \label{AdamsBallmann}
Let $X$ be a complete \cat space of finite telescopic dimension. Let $G$ be an amenable locally compact group
acting continuously on $X$ by isometries.  Then either $G$ stabilises a flat subspace (possibly reduced to a
point) or $G$ fixes a point in the ideal boundary $\bd X$.
\end{thm}

Combining this with the arguments of \cite{CapraceTD}, one obtains the following description of the algebraic
structure of amenable groups acting on \cat cell complexes.

\begin{thm}\label{cell}
Let $X$ be a \cat cell complex with finitely many types of cells and $G$ be a locally compact group admitting an
isometric action on $X$ which is continuous,  cellular and metrically proper. Then a closed subgroup $H < G$ is
amenable if and only if it is (topologically locally finite)-by-(virtually Abelian).
\end{thm}

By definition, a subgroup $H$ of a topological group $G$ is \textbf{topologically locally finite} if the closure
of every finitely generated subgroup of $H$ is compact. We refer to \cite{CapraceTD} for more details. The proof
of Theorem~\ref{cell} proceeds as in \emph{loc.~cit.} One introduces the \textbf{refined boundary} $\bdfine X$
of the \cat space and shows, using Theorem~\ref{AdamsBallmann}, that any amenable subgroup of $G$ virtually
fixes a point in $X \cup \bdfine X$; conversely any point of $X \cup \bdfine X$ has an amenable stabilizer in
$G$.

\subsection*{Minimal and reduced actions}

A basic property of \cat spaces with finite telescopic dimension is that their Tits boundary has finite
geometric dimension (see Proposition~\ref{dimension} below). Given this observation, Theorem~\ref{filtering} may
be used to extend several results of \cite[Part~I]{CapraceMonod} to the finite-dimensional case. The following
collects a few of these statements.

\begin{prop}\label{MinimalReduced}
Let $X$ be a complete \cat space of finite telescopic dimension and let $G < \Isom(X)$ be any group  of
isometries.

\begin{enumerate}
\item 
If the $G$-action is evanescent, then $G$ fixes a point in $X \cup \bd X$.

\item If $G$ does not fix a point in the ideal boundary, then there is a non-empty closed convex $G$-invariant
subset  $Y \subseteq X$ on which $G$ acts minimally.

\item Suppose that $X$ is irreducible. If $G$ acts minimally without fixed point at infinity on $X$, then so
does every non-trivial normal subgroup of $G$; furthermore, the $G$-action is reduced.

\item If $\Isom(X)$ acts minimally on $X$, then for each closed convex subset $Y \subsetneq X$ we have $\bd Y
\subsetneq \bd X$.
\end{enumerate}
\end{prop}

Following Nicolas Monod~\cite{Monod_superrigid}, we say that the action of a group $G$ on a \cat space $X$ is
\textbf{evanescent} if there is an unbounded subset $ T \subseteq X$ such that for every compact set $Q \subset
G$ the set $\{d(gx, x) : g \in Q, x \in T\}$ is bounded.  Recall further that the $G$-action is said to be
\textbf{minimal} if there is no non-empty closed convex $G$-invariant subset $Y \subsetneq X$. Finally, it is
called \textbf{reduced} if there is no non-empty closed convex subset $Y \subsetneq X$ such that for each $g \in
G$, the sets $Y$ and $g.Y$ are at bounded Hausdorff distance from one another. The relevance of the notions of
evanescent and reduced actions was first highlighted by Nicolas Monod~\cite{Monod_superrigid} in the context of
geometric superrigidity. In particular, the combination of \cite[Theorem~6]{Monod_superrigid} with
Proposition~\ref{MinimalReduced}(iii) yields the following (see \cite[Theorem~9.4]{CapraceMonod} for the
corresponding statement in the locally compact case).

\begin{cor}\label{superrigidity}
Let $\Gamma$ be an irreducible uniform (or square-integrable weakly cocompact) lattice in a product $G = G_1
\times \dots  \times G_n$ of $n \geq 2$ locally compact $\sigma$-compact groups. Let $X$ be a complete \cat
space of finite telescopic dimension without Euclidean factor. Then any minimal isometric $\Gamma$-action on $X$
without fixed point at infinity extends to a continuous $G$-action by isometries.
\end{cor}

On the other hand, combining Proposition~\ref{MinimalReduced} with Theorem~\ref{AdamsBallmann} yields the
following extension of \cite[Corollary~4.8]{CapraceMonod}.

\begin{cor}\label{radical}
Let $G$ be a locally compact  group acting continuously and minimally on a  \cat space $X$ of finite telescopic
dimension, without fixing any point at infinity. Then the amenable radical $R$ of $G$ stabilizes the maximal
Euclidean factor of $X$. In particular, if $X$ has no non-trivial Euclidean factor, then $R$ acts trivially.
\end{cor}


\subsection*{Acknowledgements}
We would like to thank Viktor Schroeder for fruitful conversations
on affine functions. We are grateful to Anton Petrunin for providing
Example~\ref{ex:Petrunin}. We thank Nicolas Monod for numerous
illuminating conversations and for pointing out that no separability
assumption on $G$ is needed for Theorem~\ref{AdamsBallmann} to hold.
Finally, thanks are due to the referee for his/her comments.

\section{Preliminaries}

\subsection{Geometric and telescopic dimension}\label{sec:dimension}
We recall some facts about dimensions of spaces with upper curvature bounds.  The \textbf{geometric dimension}
(sometimes simply called \emph{dimension}) of such spaces was defined inductively in \cite{Kleiner}, by setting
the dimension of a discrete space to be $0$ and be defining  $\dim (X)= \sup \{ \dim (S_x X) +1 | x\in X  \}$,
where $S_x X$ denotes the space of directions at the point $x$. It turns out that this notion of dimension is
closely related to more topological notions. Namely $\dim (X) \leq n$ if and only if  for all open subsets
$V\subset U$ of $X$ the relative singular homology $H_ k (U,V)$ vanishes for all $k>n$. Moreover, this is
equivalent to the fact that the topological dimension of all compact subsets of $X$ is bounded above by $n$, see
\emph{loc.~cit.}

By definition, a \cat space $X$ is said to have \textbf{telescopic dimension}~$\leq n$ if every asymptotic cone
$\lim_\omega (\vareps_n X, \star_n)$ has geometric dimension~$\leq n$. Although this will not play any role in
the sequel, we remark that the telescopic dimension is a quasi-isometry invariant. Moreover, it follows from
\cite[Th.~C]{Kleiner} that a locally compact \cat space with a cocompact isometry group has finite telescopic
dimension.

Convex subsets inherit the geometric dimension bound from the ambient space. Moreover, if $(X_i,x_i)$ is a
sequence of pointed $\mathrm{CAT}(\kappa )$ spaces of geometric dimension $\leq n$, then their ultralimit
$\lim_{\omega} (X_i,x_i)$ with respect to some ultrafilter is a $\mathrm{CAT}(\kappa )$ space of dimension at
most $n$, see \cite[Lemma 11.1]{Lytchak:rigidity}. In particular, it follows that a \cat space of geometric
dimension $\leq n$ has telescopic dimension~$\leq n$. Furthermore, we have the following.

\begin{prop}\label{dimension}
Let $X$ be a \cat space. If $X$ has telescopic dimension~$ \leq n$,
then the visual boundary $\bd X$ endowed with the Tits metric has
geometric dimension at most $n-1$.
\end{prop}
\begin{proof}
Let $o \in X$ be a base point and $C_\omega X$ be the asymptotic cone $\lim_\omega (\frac{1}{n}X, o)$. The
Euclidean cone $C(\bd X)$ embeds  isometrically into $C_\omega X$, see \cite[Lemma 10.6]{Kleiner}. Thus $\dim
(\bd X) = \dim (C (\bd X))-1 \leq \dim (C_\omega X) \leq n$.
\end{proof}

We emphasize that a \cat space $X$ may have finite-dimensional Tits boundary without being of finite telescopic
dimension, even if $X$ is proper. Indeed, consider for instance the positive real half-line and glue at each
point $n \in \mathbb N$ an $n$-dimensional Euclidean ball of radius $n$. The resulting space is proper and \cat,
its ideal boundary consists of a single point, but each of its asymptotic cones contains an infinite-dimensional
Hilbert space.


\medskip%
We shall use a topological version Helly's classical theorem that holds in much greater generality (see
\cite{Dugundji} as well as \cite[\S3]{Farb:Helly} for a related discussion). The following statement is an
immediate consequence of \cite[Proposition~5.3]{Kleiner} since intersections of convex sets are either empty or
contractible.

\begin{lem} \label{kleiner}
Let $X$ be a \cat space of geometric dimension~$\leq n$. Let $\{ U_{\alpha}\}_{\alpha\in A }$ be a finite family
of open convex subsets of $X$. If for each subset $B \subset A$ with at most $n+1$ elements the intersection
$\bigcap_{\alpha \in B} U_{\alpha}$ is non-empty, then $\bigcap _{\alpha \in A} U_{\alpha}$ is non-empty.\qed
\end{lem}

\subsection{Inner points}  \label{subsecinner}
Following \cite{LytchakSchroeder}, we shall say that a point $o$ of a \cat space $X$ is  a \textbf{topologically
inner} point if $X\setminus \{o \}$ is not contractible. For each topologically inner point there is some
$\vareps >0$ and a compact subset $K$ of $X$ with $d(o,K)\geq \vareps$ with the following property: For each
$x\in X$ there is some $\bar x \in K$ such that $xo\bar x$ is a geodesic. Thus every geodesic segment which
terminates at $o$ may be locally prolonged  beyond $o$; in loose terms, the space $X$ is \emph{geodesically
complete at the point $o$}. In a \cat space which is locally of finite geometric dimension, the set of
topologically inner points is dense, see \cite[Theorem~1.5]{LytchakSchroeder}. In particular it is non-empty.

\section{Jung's theorem} \label{secjung}

\subsection{\cat case}\label{sec:Jung:cat0}
Throughout the paper, we shall adopt the following notational convention. Given a subset $Y \subseteq X$ we
denote the distance to $Y$ by $d_Y$, namely $d_Y : X \to \RR : x \mapsto \inf_{y \in Y} d(x, y)$. We further
recall that the \textbf{intrinsic radius} of a subset $Z$ of a metric space $X$ is defined as
$$\rad(Z) = \inf_{z \in Z} \{r \in \RR_{>0} \; | \; Z \subseteq B(z, r) \}.$$
This notion should not be confused with the \textbf{circumradius} (or \textbf{relative radius}), defined as
$$\rad_X(Z) = \inf_{x \in X} \{r \in \RR_{>0} \; | \; Z \subseteq B(x, r) \}.$$

\smallskip
Bounded closed convex subsets of non-positively curved spaces have the \textbf{finite intersection property}
(see \cite[Proof of Theorem~B]{LangSchroeder} or \cite[Theorem~14]{Monod_superrigid}). This means that for any
family $\{ X_{\alpha} \}_{\alpha \in A}$ of \emph{bounded} closed convex  subsets of a \cat space $X$ the
intersection $\bigcap_{\alpha \in A} X_\alpha$ is non-empty whenever the intersection of each finite sub-family
is non-empty.

\begin{lem} \label{nplus}
Let $X$ be a \cat space of geometric dimension~$\leq n$ and $Y \subseteq X$ be a subset of finite diameter.  If
for all subsets $Y'\subseteq Y$ of cardinality $|Y'| \leq n+1$ we have $\rad_X (Y') \leq r$ then $\rad_X (Y)
\leq r$.
\end{lem}

\begin{proof}
Fix an arbitrary $r'>r$. For $y\in Y$, denote by $O_y$ the open ball of radius $r'$ around  $y$. These balls are
convex and, by assumption, the intersection of any collection of at most $(n+1)$ such balls is non-empty. By
\lref{kleiner} the intersection of any finite collection of such balls is non-empty. Since $r'>r$ is arbitrary,
this implies that each finite subset $Y'$ of $Y$ has radius at most $r$.  For $y\in Y$, denote now by $B_y$ the
closed ball of radius $r$ around $y$. Then the intersection of each finite collection of $B_y$ is non-empty,
hence the intersection of all $B_y$ is non-empty.  For any point $x$ in this intersection, we get $d(x,y)\leq r$
for all $y\in Y$. Hence $\rad_X (Y) \leq r$.
\end{proof}

\begin{proof}[Proof of Theorem~\ref{raddiam}]
Theorem~A from \cite{LangSchroeder} ensures that for any \cat space $X$ and each subset $Y \subset X$ of
cardinality at most $n+1$, the inequality $ \rad_X (Y) \leq \sqrt {\frac {n} {2(n+1)}} \diam(Y)$ holds. In view
of this, it follows from \lref{nplus} that the inequality $ \rad_X (Y) \leq \sqrt {\frac {n} {2(n+1)}} \diam(Y)$
holds for any subset $Y$ of a \cat space $X$ of geometric dimension~$\leq n$.

Assume conversely that $X$ has geometric dimension~$> n$. By \cite[Theorem~7.1]{Kleiner}, there exist a sequence
$(\lambda_k)$ of positive real numbers such that $\lim_k \lambda_k = \infty$, and a sequence $(Y_k, \star_k)_{k
\geq 0}$ of pointed subsets of $X$ such that
$$\lim_\omega (\lambda_k Y_k, \star_k) =  \RR^{n+1}$$
for any non-principal ultrafilter $\omega$. We may then find $n+2$ sequences $(y_k^i)_{k\geq 0}$ of points of
$Y_k$ indexed by $i \in \{0, 1, \dots, n+1\}$ such that the set $\Delta = \{ \lim_\omega (y_k^i) \; | \; i =0,
\dots, n+1\} $ coincides with the vertex set of a regular simplex of diameter~$1$ in $\RR^{n+1}$. Since the
equality case of the $(n+1)$-dimensional Jung inequality is achieved in the case of $\Delta$, we deduce that
there exists some $k \geq 0$ such that the $n$-dimensional Jung inequality fails for the subset $\Delta_k =
\{y_k^i \; | \; i =0, \dots, n+1\} \subset X$.

\medskip%
Assume now that $X$ has telescopic dimension~$\leq n$ and suppose for a contradiction that for some fixed
$\delta >0$ and  for each integer $k>0$ there is a subset $Y_k \subset X$  such that $\diam(Y_k)>k$ and
$\rad_X(Y_k) \geq (\sqrt{\frac{n}{2(n+1)}} +\delta ) \diam(Y_k)$. Let $\star_k$ be the circumcentre of $Y_k$.
Setting $\lambda_k = \rad_X(Y_k)$, it then follows that the asymptotic cone $\lim_\omega(\frac{1}{\lambda_k} X,
\star_k)$ possesses a subset $\lim_\omega(Y_k)$  which fails to satisfy the $n$-dimensional Jung inequality. The
contradicts the first part of the statement which has already been established.

Assume conversely that $X$ has telescopic dimension~$> n$. Then, by~\cite[Theorem~7.1]{Kleiner} there exists a
sequence $(Y_k, \star_k)_{k \geq 0}$ of pointed subsets of $X$ such that $\lim_\omega (Y_k, \star_k) =
\RR^{n+1}$. In particular $\diam(Y_k)$ tends to $\infty$ with $k$ and we conclude by the same argument as
before.
\end{proof}

\subsection{Rigidity and other curvature bounds}
In this subsection, we briefly sketch  the analogues of \tref{raddiam} in the case of non-zero curvature bounds
and address the equality case.  Since the results are not used in the sequel, we do not provide complete proofs.

Following word by word the proof of \tref{raddiam}  and using the results of  \cite{LangSchroeder} for other
curvature bounds, one obtains the following.

\begin{prop}
Let $X$ be a $\mathrm{CAT}(-1)$ space of geometric dimension at most $n$. Let $Y$ be a subset of $X$ of diameter
at most $D$. Then the radius of $Y$ in $X$ is bounded above by $r_n (D)$, where $r_n(D)$ denote the radius of
the regular $n$-dimensional simplex $\Delta _D$ in the $n$-dimensional real hyperbolic space $\mathbb H ^n$ of
diameter $D$.\qed
\end{prop}

In the positively curved case one needs to assume a bound on the radius in order for the balls in question to be
convex. An additional technical difficulty arises from the fact the the whole space may  be non-contractible in
this case, and the statement of \lref{kleiner} has therefore to be slightly modified in that case. The resulting
radius--diameter estimate is the following.

\begin{prop}
Let $X$ be a $\mathrm{CAT}(1)$ space of dimension  $\leq n$. Let $Y$ be a subset of $X$ of circumradius $r
<\frac {\pi} 2$.   Then the diameter of $Y$ is at least $s_n (r)$, where $s_n(r)$ is the diameter of the regular
simplex of radius $r$ in the round $\mathbb S ^n$.
\end{prop}

\begin{rem}
In a similar way it can be shown, that  the assumption $r=\rad_X (Y)< \frac {\pi} 2$ is fulfilled as soon as
$\diam (Y) < k_n= \arccos (-1/(n+1)) $.
\end{rem}

It is shown in  \cite{LangSchroeder} that for a subset $Y$ of cardinality $\leq n+1$, the equality in
\tref{raddiam} holds if and only if the convex hull of these points is isometric to a regular Euclidean simplex.
Arguing as in the proof of \tref{raddiam} one obtains that if $X$ is locally compact, the inequality becomes an
equality if and only if the convex hull of $Y$ contains a regular  $n$-dimensional Euclidean simplex of diameter
equal to the diameter of $Y$. If $X$ is not locally compact the same statement holds for the convex hull of the
ultraproduct $Y^{\omega} \subset X ^{\omega}$. Similarly, the analogous rigidity statements hold for spaces with
other curvature bounds for the same reasons.

\section{Convex functions and their gradient flow}

\subsection{Gradient flow}\label{sec:gradient}
We recall some basics about gradient flows associated to  convex functions. We refer to \cite{Mayer} for the
general case and to \cite{Lytchak:open} for the simpler case of Lipschitz continuous functions; only the latter
is relevant to our purposes.

Given a \cat space $X$, a map $f : X \to \mathbf{R}$ is called \textbf{convex} if its restriction $f\circ
\gamma$ to each geodesic $\gamma$ is convex. Basic examples of convex functions on \cat spaces are distance
functions to points or to convex subsets, and Busemann functions, see \cite[II.2 and~II.8]{Bridson-Haefliger}.

Let $f$ be a continuous convex  function on a \cat space $X$. For a point $p\in X$, the \textbf{absolute
gradient} of the concave function $(-f)$ at $p$ is defined by the formula

$$|\nabla_p (-f) |= \max \Big\{ 0, \limsup_{x \to p}\frac{f(p) - f(x)}{d(p, x)} \Big\}. $$

The absolute gradient is a non-negative, possibly infinite function. It is bounded above by the Lipschitz
constant if $f$ is Lipschitz continuous. A fundamental object attached to the function $f$ is the
\textbf{gradient flow} which consists of a map $\phi :[0,\infty )\times  X \to X$ which, loosely speaking, has
the property that $\phi_0 = \id$ and $\phi_t(x)$ follows for each $x$ the path of steepest descent of $f$   from
$x$. The gradient flow is indeed a flow in the sense that it satisfies $\phi_{s+t}(x) = \phi_s \circ \phi_t(x)$
for all $x\in X$. The most important property of gradient flows, originally observed by Vladimir
Sharafutdinov~\cite{Sharaf} in the Riemannian context, is that the flow $\phi _t$ is \textbf{semi-contracting}.
In other words, for each $t\geq 0$, the map $\phi _t: X \to X$ is $1$-Lipschitz (see
\cite[Theorem~1.7]{Lytchak:open}). We refer to \cite{Mayer} or \cite[\S9]{Lytchak:open} for more details and
historical comments.

\begin{rem}
Originally, the gradient lines and flows were defined for concave functions by Sharafutdinov~\cite{Sharaf} in
the case of manifolds; they are also commonly used for semi-concave  (but not semi-convex) functions. Moreover
the gradient usually represents the direction of the maximal \emph{growth} of the function rather than its
maximal \emph{decay}. This explains the slightly cumbersome notation $|\nabla _x (-f)|$ that we use here.
\end{rem}

For each $x\in X$ the \textbf{gradient curve} $t\mapsto \phi_t(x)$ of $f$  has the following properties  (and is
uniquely characterised by them).

\renewcommand{\theenumi}{\arabic{enumi}}
\begin{enumerate}
\item The curve $t \mapsto \phi _t (x)$ has velocity $|\phi_t (x)  '| =|\nabla _{\phi _t(x)} (-f)|$ for almost
all $t \geq 0$.

\item The restriction $t\mapsto f(\phi_t (x))$ of $f$ to the gradient curve is convex. Furthermore it satisfies
$(f\circ \phi_t (x))' = -|\nabla _{\phi _t(x)} (-f)|^2$.
\end{enumerate}

We define the   \textbf{velocity of escape} of the flow $\phi_t$ at the  point $x \in X$ by
$$\limsup_{t \to \infty} \frac{d(x, \phi_t(x))}{t}.$$
Since the flow $\phi_t$ is semi-contracting, the lim sup in the above definition may be replaced by a usual
limit. Moreover, it does not depend on the starting point $x$. The following statement is an application of the
main result of \cite{KarlssonMargulis} (to a deterministic setting).

\begin{prop} \label{karlsson}
Let $f$ be a convex Lipschitz function on  a \cat space $X$. If $\vareps =\inf _{x\in X} |\nabla _x  (-f)| >0$
then there is a unique point $\xi_f \in \partial X$ such that  for all $x\in X$ the gradient curve $\phi _t (x)$
defined by $f$ converges to $\xi_f$ for $t\to \infty$.
\end{prop}

\begin{rem}
In particular, the existence of a function $f$ as in \pref{karlsson} implies that the ideal boundary of $X$
is non-empty.
\end{rem}

The following construction due to Anton Petrunin shows that the conclusion of Proposition~\ref{karlsson} fails
without a uniform lower bound on the absolute gradient.

\begin{ex}\label{ex:Petrunin}
Choose an acute  angle in $\RR^2$ enclosed by two rays $\gamma ^{\pm} (t) =t \cdot v^{\pm}$ emanating from the
origin. Let $f_n (w)=\la w,x_n \ra$ be linear maps on $\RR^2$  such that the following conditions hold. First,
for all $n$, we require that $\la x_n, v^{\pm }\ra$ be positive. For odd  (resp. even)  $n$, the direction $v^+$
(resp. $v^-$) is between $x_n$ and $v^-$ (resp. $x_n$ and $v^+$).  Moreover, the sequence $(x_n)$ satisfies the
recursive condition $\la x_n,v^- \ra = \la x_{n-1},v^+\ra$ for even $n$ and $\la x_n,v^+ \ra= \la x_{n-1},v^-
\ra$ for odd $n$. Finally, we require that the length $\|x_n\|$ tends to $0$ as $n$ tends to infinity. It is
easy to see that such a sequence $(x_n)$ exists.

Now let $p_1=v^-$ and define inductively $p_n$ on $\gamma ^+$ (resp. $\gamma ^-$) for $n$ even (resp. odd) to be
the point such that $p_n-p_{n-1}$ is parallel to $x_n$.  This just means that  $p_n$ arises from $p_{n-1}$ by
following the gradient flow of the affine (and hence concave) function $f_n$.

Define the numbers $C_n$ by $C_0=0$ and $f_n(p_{n+1})-f_{n+1} (p_{n+1}) = C_{n+1}- C_n$. Consider the concave
function $f(x)=\inf (f_n (x) +C_n)$. One verifies that on the geodesic segment $(p_n, p_{n+1})$ the function $f$
coincides with $f_n$ (in fact on a neighbourhood of all points except $p_{n+1}$). Hence the segment joining
$p_n$ to $p_{n+1}$ is part of a gradient curve of $f$. Therefore the appropriately parametrised piecewise
infinite geodesic $\gamma$ running through all $p_i$ is a gradient curve of $f$. It is clear that both $v^-$ and
$v^+$ (and all unit vectors between them) considered as points in the ideal boundary are accumulation points of
$\gamma$ at infinity.
\end{ex}

\begin{proof}[Proof of Proposition~\ref{karlsson}]
From the assumption that $\vareps =\inf _{x\in X} |\nabla _x (- f)| >0$, we deduce that $f(\phi _t(x))-f(x) \leq
-\vareps ^2 t$ for all $x\in X$. In view of Property~(2) of the gradient curve recalled above and the fact that
$f$ is Lipschitz, we deduce that the velocity of escape of the gradient curve is strictly positive.

An important consequence of \cite[Theorem 2.1]{KarlssonMargulis} is that any semi-contracting map $F:X\to X$ of
a complete \cat space $X$ with strictly positive velocity of escape  $ \limsup_{n \to \infty} \frac{d(p,
F^n(p))}{n}$ has the following convergence  property: There is a unique point $\xi_F$ in the ideal boundary
$\partial X$ of $X$, such that for all $p\in X$ the sequence $p_n =F^n (p)$ converges to $\xi_F$ in the cone
topology. In view of the above discussion, we are in a position to apply this result to $F=\phi_1$, from which
the desired conclusion follows.
\end{proof}

\subsection{Asymptotic slope and a radius estimate}

Finally we recall an observation of Eberlein (\cite{Eberlein}, Section 4.1) about the size of the set of points
in the ideal boundary with negative asymptotic slopes.

Let  $f:X\to \RR$ be a continuous convex function. For each geodesic ray $\gamma:[0,\infty ) \to X $ one defines
the \textbf{asymptotic  slope}  of $f$ on $\gamma$ by $\lim_{t\to \infty} (f\circ \gamma '(t) )$. This defines a
number in $(-\infty, +\infty]$ which depends only on the point at infinity $\gamma(\infty) \in \bd X$. Thus one
obtains a function $slope _f: \bd X \to (-\infty, +\infty]$. One says that a point $\xi \in \bd X$ is
\textbf{$f$-monotone} if $slope _f (\xi) \leq 0$. This is equivalent to saying that the restriction of $f$ to
any ray asymptotic to $\xi$ is non-increasing. One denotes the set of all $f$-monotone points by
$X_f (\infty ).$

\begin{lem} \label{eberle}
Let $f$ be a convex Lipschitz function on  a complete \cat space $X$ such that $\inf _{x\in X} |\nabla _x (-f)|
>0$.
Then for each point $\xi \in X_f(\infty)$, we have $d_{\mathrm{Tits}} (\xi, \xi_f )  \leq \frac \pi 2$, where
$\xi_f$ is the canonical point provided by \pref{karlsson}.
\end{lem}

\begin{proof}
Eberlein's argument for the proof of \cite[Proposition 4.1.1]{Eberlein} (which is also reproduced in the proof
of \cite[Theorem~1.1]{Nagano}) shows, that for any $p\in X$ and any sequence $t_i$, such that $\phi _{t_i} (p)$
converges to some point $\xi \in \partial X$, the Tits-distance between $\xi$ and any other point $\psi \in X$
is at most $\frac \pi 2$.
\end{proof}

\subsection{The space of convex functions} \label{spacefunction}
Pick a base point $o\in X$. Denote by $\mathcal{C}_0$ the set of all $1$-Lipschitz convex functions $f$ on $X$
with $f(o)=0$. We view it as subset of the locally convex topological vector space $\mathcal B$ of all functions
$f$ on $X$ with $f(o)=0$, where the latter is considered with the topology of pointwise convergence.  The subset
$\mathcal{C}_0$ may thus be considered as a closed subset of the infinite product  $\prod _{x\in X} I_x$, where
$I_x$ is the interval $I_x =[-d(o,x),d(o,x)]$. Since a convex combination of convex $1$-Lipschitz functions is
convex and $1$-Lipschitz, the set $\mathcal{C}_0$ is a convex compact subset of $\mathcal B$.

The isometry group $G = \Isom(X)$ of $X$ acts continuously on $\mathcal B$ by $g\cdot f: x \mapsto f(gx)-f(go)$
and preserves the subset $\mathcal{C}_0$.  Consider the map $i:X\to \mathcal{C}_0$ given by $i(x) := \bar d _x$,
where $\bar d_ x$ is the normalized distance function $\bar d _x (y)= d(x,y)-d(x,o)$. Note that the map $i$ is
$G$-equivariant. In particular, the subset $\mathcal C=i(X) \subset \mathcal{C}_0$ as well as its closure and
closed convex hull are $G$-invariant.  If $X$ is locally compact, then $\overline{\mathcal C}$ consists
precisely of normalized distance and Busemann functions on $X$, and is thus nothing but the visual
compactification of $X$. However, if $X$ is not locally compact, then $\bar {\mathcal C}$ may be much larger,
and the convergence in $\mathcal{C}_0$ may be rather strange.

\begin{ex}
Let $X'$ be a separable Hilbert space with origin $o$ and an orthonormal base $\{e_n\}_{n \geq 0}$. Then the
sequence $\bar d_{ne_n}$ converges in $\mathcal{C}_0$ to the constant function.
\end{ex}

\begin{ex}
Let $X''$ be a metric tree consisting of a single vertex $o$ from which emanate countably many infinite rays
$\eta_n$. In other words $X''$ is the Euclidean cone over a discrete countably infinite set. Let $b_n$ denote
the Busemann function associated with $\eta _n$. Then $b_n$ converge in $\mathcal{C}_0$ to the distance function
$d_o$.
\end{ex}

We emphasize that the choice of the base point $o$ does not play any role: any change of base point amounts to
adding an additive constant.

In some sense, the set $\mathcal C$ may serve in the non-locally compact case as a generalized ideal boundary.
It is therefore important to understand how ``large'' it really is. This will be the purpose of the next
subsection.

\subsection{Affine functions on spaces of finite telescopic dimension}

Recall that a function $f:X\to \RR$ is called \textbf{affine} if its restriction to any geodesic is affine.
Equivalently, for all pairs $x^+,x^- \in X$ with midpoint $m$ we have $f(x^+) +f(x^-)=2 f(m)$. A simple-minded
but noteworthy observation is that affine functions are precisely those convex functions $f$ whose opposite
$(-f)$ is also convex. Clearly, constant functions are affine; thus any \cat space admit affine functions.
However, the very existence of \emph{non-constant} affine functions imposes very strong restrictions on the
underlying space, see~\cite{LytchakSchroeder}. The following result also provides an illustration of this
phenomenon, which will be relevant to the proof of Theorem~\ref{AdamsBallmann}.

\begin{prop}\label{LytchakSchroeder}
Let $X$ be a \cat space of finite telescopic dimension which is not reduced to a single point and such that
$\Isom(X)$ acts minimally. If $\overline{ \mathcal{C}}$ contains an affine function, then there is a splitting
$X = \RR \times X'$.
\end{prop}

Recall from \cite[II.6.15(6)]{Bridson-Haefliger} that any complete \cat space $X$ admits a canonical splitting
$X=\mathbb{E} \times X'$ preserved by all isometries, where $\mathbb{E}$ is a (maximal) Hilbert space called the
\textbf{Euclidean factor} of $X$. It is shown in \cite[Corollary~4.8]{LytchakSchroeder} that if $X$ is locally
finite-dimensional and if $\Isom(X)$ acts minimally, then $X'$ does not admit \emph{any} non-constant affine
function. The main technical point in the proof of the latter fact is the existence  of inner points (see
\S\ref{subsecinner}).

In order to deal with the case of \emph{asymptotic} dimension bounds, we need to substitute this by some coarse
equivalent. This substitute is provided by Lemma~\ref{lem:Q}, which is of technical nature. In the special case
of spaces of finite \emph{geometric} dimension, it follows quite easily from the existence of inner points ;
therefore, the reader who is only interested in those spaces may wish to skip it.

\begin{lem}\label{lem:Q}
Let $X$ be an unbounded space of finite telescopic dimension. 
Then
there are sequences of positive numbers $D_j \to \infty$, $\delta_j \to 0$ and sequences of points $p_j \in X$
and of finite subsets $Q_j \subset X$ with the following two properties.
\begin{itemize}
\item[(1)] $Q_j$ is contained in the ball of radius $D_j (1+\delta _j)$ around $p_j$.

\item[(2)] For all $s \in X$, there is some $q_j \in Q_j$ with $d(s, q_j)-d(s,p_j) \geq D_j$.
\end{itemize}
\end{lem}

\begin{proof}
%
Consider $\tilde X= \lim _{\omega} (\frac 1 n X, o)$ and let $\tilde p =(p_n)$ be  an inner point of
$\tilde{X}$. Let $\vareps>0$ and the compact subset  $K\subset \tilde X$ be chosen as in \S\ref{subsecinner}.
Moving points of $K$ towards $\tilde p$ we may assume that all point of $K$ have distance $\vareps$ to $\tilde
p$. Furthermore, there is no loss of generality in assuming $\vareps < 1$.

Since $K$ is compact, there exist finite subsets $Q_n \in \frac 1 n X$ with $\lim _{\omega } Q_n =K$ and
$d(p_n,q)\leq \vareps n$ for all $q\in Q_n$.

In view of the defining property of $K$, we deduce that for all $\delta \in (0, \vareps )$ and all $n_0$, there
is some $n = n(\delta, n_0)> n_0$ such that for any $s \in X$ with $d(s, p_n) \leq n$, there is some $q \in Q_n$
with $d(s, q) \geq d(s, p_n) + n (\vareps - \delta)$.

Assume now  $\delta \in (0, \frac \vareps 2)$. Given $\tilde s \in X$ with $d(\tilde s ,p_n)\geq n$ and choose
the point $s$ between $p_n$ and $\tilde s$ with $d(p_n,s)=n$. Let $q \in Q_n$ be such that $d(s, q) \geq d(s,
p_n) + n (\vareps - \delta)$. Using the \emph{law of cosines} in a comparison triangle for $\Delta(\tilde s,
p_n, q)$, we deduce from that \cat inequality that
$$\frac{d(\tilde s, q)^2 - d(\tilde s, p_n)^2 - d(p_n, q)^2}{d(\tilde s, p_n)} \geq
\frac{n^2(\vareps - \delta)(2 + \vareps - \delta) - d(p_n, q)^2}{n}.$$
Since $d(p_n, q) \leq \vareps n$ and $d(\tilde{s}, p_n) = n + d(s, \tilde s)$, we deduce
$$\begin{array}{rcl}
d(\tilde s, q)^2 - d(\tilde s, p_n)^2 & \geq & n^2(\vareps - \delta)(2 + \vareps - \delta) + 2 n
( \vareps - \delta(1 + \vareps) + \frac{\delta^2}  2) d(s, \tilde s)\\
& \geq &  n^2(\vareps - 2 \delta)(2 + \vareps - 2 \delta) + 2n (\vareps - 2 \delta) d(s, \tilde s) \\
& = & n^2(\vareps - 2 \delta)^2 + 2n(\vareps - 2\delta) d(\tilde s, p_n).
\end{array}
$$
This implies that $d(\tilde s,q) \geq d(\tilde s, p_n) + n(\vareps- 2 \delta )$.

\medskip%
It remains to define the desired sequences by making the appropriate choices of indices. This may be done as
follows. For each integer $j
>0$, we now set $\delta_j = \frac{\vareps}{2j}$ and $n_j = n(\frac{\vareps \delta_j}{4}, j)$, where
$n: (0, \vareps) \times \mathbf{N} \to \mathbf{N}: (\delta, n_0) \mapsto n(\delta, n_0)$ is the map considered
above. Finally we set $p_j = p_{n_j}$, $Q_j = Q_{n_j}$ and $D_j =\vareps n_j (1 - \frac{\delta_j} 2)$.
\end{proof}

\begin{proof}[Proof of Proposition~\ref{LytchakSchroeder}]
Assume that the set $\mathcal{A}$ of affine functions contained in $\overline{\mathcal{C}}$ is nonempty. For
each integer $j >0$ we set
$$C_j = \{ p \in X \; | \; \forall f \in \mathcal{A} \; \exists z \in X, f(z)-f(p)=D_j \text{ and }
d(p,z)\leq (1+\delta_j)D_j\},$$
where $(D_j)$ and $(\delta_j)$ are the sequences provided by Lemma~\ref{lem:Q}.

\smallskip \noindent%
\emph{We claim that $C_j$ is non-empty.}

\smallskip%
Let us fix the index $j$ and consider the point $p_j \in X$ provided
by Lemma~\ref{lem:Q}. We will show that $p_j$ belongs to $C_j$. To
this end, let $f \in \mathcal A$. By definition, there is a sequence
$(s_n)$ of points of $X$ such that the sequence $(\bar d_{s_n})$
converges pointwise to $f$, where $\bar d_s$ denotes the normalised
distance function $\bar d_s(\cdot) = d(s, \cdot) - d(o,s)$. Now for
each $n$, Lemma~\ref{lem:Q} provides a point $q_n \in Q_j$ such that
$d( p_j, q_n) \leq D_j(1 + \delta_j)$ and $\bar d_{s_n}(q_n) - \bar
d_{s_n}(p_j) = d_{s_n}(q_n) -  d_{s_n}(p_j) \geq D_j$. Upon
extracting, we may assume that $(q_n)$ is constant and is equal to
$q \in Q_j$, since $Q_j$ is finite. Now, passing to the limit as $n
\to \infty$, we obtain $f(q) - f(p_j) \geq D_j$ and $d(p,q) \leq
(1+\delta_j)D_j$. Finally, since $f$ is affine, there exists a
unique point $z \in [p_j, q]$ such that $f(z) - f(p_j) = D_j$. This
confirms that $p_j \in C_j$ and the claim stands proven.

\smallskip \noindent%
\emph{We claim that $C_j$ is convex.}

\smallskip%
Indeed, let $p_1, p_2 \in C_j$, $f \in \mathcal{A}$ and $z_1, z_2$ such that for $i = 1, 2$ we have  $f(z_i) -
f(p_i) =D_j$ and $d(p_i, z_i )\leq (1+\delta_j)D_j$. Given $p \in [p_1, p_2]$ at distance $\lambda  d(p_1, p_2)$
from $p_1$, where $\lambda \in (0,1)$, let $z \in [z_1, z_2]$ be the unique point at distance $\lambda d(z_1,
z_2)$ from $z_1$. Since the distance function is convex, we have $d(p, z) \leq (1+\delta_j)D_j$. Furthermore,
since $f$ is affine we have $f(z) - f(p) = D_j$. Hence $p \in C_j$.

\smallskip \noindent%
\emph{We claim that $C_j= X$.}

\smallskip%
Since $\mathcal{A}$ is $\Isom(X)$-invariant, so is $C_j$. In view of the assumption of minimality on the
$\Isom(X)$-action, it follows that $C_j$ is dense in $X$ for all $D_j, \delta_j>0$. Now the claim follows from a
routine continuity argument using the fact that $f$ is $1$-Lipschitz.

\medskip%
For each $f \in \mathcal{A}$ and $p \in C_j$, we have $|\nabla_p(f)| \geq \frac 1 {1+\delta_j}$ since $f$ is
affine. By the previous claim, the latter inequality holds for all $p \in X$ and all $\delta_j>0$. Since $f$ is
$1$-Lipschitz it follows that $|\nabla_p(f)| = 1$ for all $p \in X$. Now Proposition~\ref{karlsson} yields a
point $\xi \in \bd X$ to which the gradient curve $t \mapsto \phi_t(p)$ converges as $t$ tends to infinity.
Since the gradient curve as velocity~$1$ (see \S\ref{sec:gradient}) we deduce that it is a geodesic ray pointing
to $\xi$. It follows that $-f$ is a Busemann function associated with $\xi$. In particular $-f = \lim_n \bar
d(\phi_n(p), \cdot)$ belongs to $\overline C$, hence to $\mathcal{A}$ since $f$ is affine. This yields another
point $\xi' \in \bd X$ and a geodesic ray $\phi'_t(p)$ pointing to $\xi'$. The concatenation of both rays is a
geodesic line $\gamma$ joining $\xi'$ to $\xi$ such that $(f\circ \gamma)'=1$. At this point,
\cite[Lemma~4.1]{LytchakSchroeder} yields the desired splitting.
\end{proof}

We conclude this section with a technical property of the space of functions $\mathcal C_0$ valid for arbitrary
\cat spaces.

\begin{lem} \label{nonaffine}
Let $X$ be any \cat space and $\mathcal{C}_0$ be as above.  Given a compact subset $\mathcal A \subset
\mathcal{C}_0$, if $\mathcal A$ does not contain any affine function, then the closed convex hull $Conv
(\mathcal A)$ does not contain any affine function either.
\end{lem}

\begin{proof}
For any $f\in \mathcal A$  we find some pair of points $x_f^+,x_f^- \in X$ with midpoint $m_f$, such that
$\vareps _f = f(x_f^+) +f(x_f ^-)- 2f(m_f) >0$.  For each $f\in \mathcal A$, let $U_f$ be the open subset of
$\mathcal A$ consisting of all $h$ with $h(x_f^+) +h(x_f ^-)- 2h(m_f) >\vareps _f /2$. Since $\mathcal A$ is
compact, finitely many $U_{f_i}$ cover $\mathcal A$. Therefore, using the convexity of $h$, we deduce that
$$r(h) := \sum_i \big( h(x_{f_i}^+) +h(x_{f_i} ^-)- 2h(m_{f_i})\big)\geq \inf \{\vareps _{f_i} \} >0$$
for all $h\in \mathcal A$. Thus the continuous functional $h \mapsto r(h)$ is strictly positive on $\mathcal A$,
hence it is positive on the compact convex hull of $\mathcal A$. Therefore each $f\in Conv (\mathcal A)$ is
non-affine on at least one of the geodesics $[x_{f_i} ^+, x_{f_i}^-]$.
\end{proof}

\section{Filtering families of convex sets}

The purpose of this section is to prove Theorem~\ref{filtering}. We start by considering an analogous property
for finite-dimensional \catone spaces.

\subsection{\catone case}

We start with the following analogue of the finite intersection property for bounded convex sets in \cat spaces.

\begin{lem}\label{lem:cat1:FiniteIntersection}
Let $X$ be a complete \catone space of radius~$<\frac{\pi}{2}$. Then any filtering family $\{X_\alpha\}_{\alpha
\in A}$ of closed convex subspaces has a non-empty intersection.
\end{lem}

\begin{proof}
Given \cite[II.2.6(1) and II.2.7]{Bridson-Haefliger}, the proof is identical to that in
\cite[Theorem~14]{Monod_superrigid}.
\end{proof}

\begin{lem}\label{lem:cat1:nested}
Let $X$ be a finite-dimensional \catone space and $\{X_i\}_{i \geq 0}$ be a decreasing sequence of closed convex
subsets such that $\rad(X_i)\leq \frac \pi 2$. Then the intersection $\bigcap_i X_i$ is a non-empty subset of
intrinsic radius~$\leq \pi /2$.
\end{lem}

\begin{proof}
Let $z_i$ be a centre of $X_i$ and $Z = \{z_i \; | \; i \geq 0\}$. By assumption $d(z_i,z_j) \leq \frac \pi 2$
for all $i, j$. Since any ball of radius~$\leq \frac \pi 2$ is convex, it follows that the closed convex hull
$C$ of $Z$ has intrinsic radius~$\leq \frac \pi 2$.

We claim that $\rad(C) < \frac \pi 2$. Otherwise we have $\rad(C) = \frac \pi 2$ and every $z \in Z$ is a centre
of $C$. Since the set of all centres is \emph{convex}, it follows that every point of $C$ is a centre. This
implies $\diam(C) \leq \rad (C)$, which contradicts \cite[Proposition~1.2]{BalserLytchak_Centers} and thereby
establishes the claim.

Let $C_i$ be the convex hull of $\{z_j \; | \; j\geq i\}$. Then $(C_i)_{i \geq 0}$ is a decreasing sequence of
closed convex subsets in a \catone space of radius $<\pi /2$. By Lemma~\ref{lem:cat1:nested}, the intersection
$Q = \bigcap_i C_i$ is non-empty. Notice that $C_i \subseteq X_i$ whence $Q \subseteq \bigcap_i X_i$. The latter
intersection is thus non-empty.

For each $x\in \bigcap_i X_i$ we have $d(x, z_j)\leq \frac \pi 2$ for all $j$. Thus $C_j$  is contained in the
ball of radius~$\frac \pi 2$ around $x$. Therefore $d(x,q)\leq \pi /2$ for all $x\in \bigcap_i X_i$ and $q\in
Q$. This shows that $\bigcap_i X_i$ has radius at most $\frac \pi 2$.
\end{proof}

\begin{prop}\label{prop:cat1:filtering}
Let $X$ be a finite-dimensional \catone space and $\{X_\alpha\}_{\alpha \in A}$ be a filtering family of closed
convex subsets such that $\rad(X_\alpha)\leq \frac \pi 2$ for each $\alpha \in A$. Then the intersection
$\bigcap_{\alpha \in A} X_\alpha$ is a non-empty subset of intrinsic radius~$\leq \pi /2$.
\end{prop}

\begin{proof}
We proceed by induction on $n = \dim X$. There is nothing to prove in dimension $0$, hence the induction can
start.

If $\dim(X_0) < n$ for some index $0 \in A$, then the induction hypothesis applied to the filtering family
$\{X_0 \cap X_\alpha\}_{\alpha \in A}$ yields the desired conclusion. We assume henceforth that $\dim(X_\alpha)
= n$ for each $\alpha \in A$.

For $\beta \in A$, let $z_\beta$ be a centre of $X_\beta$. If $d_{X_\alpha}(z_\beta)=\frac{\pi}{2}$ for some
$\beta \in A$, then the closed convex hull of $z_\beta$ and $X_\alpha$ coincides with the spherical suspension
of $z_\beta$ and $X_\alpha$ (see \cite[Lemma~4.1]{Lytchak:rigidity}) and hence has dimension $1+\dim(X_\alpha)$.
This is absurd since $\dim(X_\alpha) = \dim(X)$. We deduce $d_{X_\alpha}(z_\beta) <\frac{\pi}{2}$ for all
$\alpha, \beta \in A$.

\smallskip%
Assume now that $\sup_{\alpha \in A} d_{X_{\alpha}}(z_\beta) =\frac \pi 2$. Then there is a countable sequence
$(X_{\alpha_i})_{i \geq 0}$ with $\alpha_i \in A$ such that $\lim_i d_{X_{\alpha_i}}(z_\beta) = \frac \pi 2$.
Upon replacing $X_{\alpha_j}$ by $\bigcap _{i=0} ^j X_{\alpha_i}$ we may and shall assume that the sequence
$(X_{\alpha_i})_{i \geq 0}$ is decreasing. By Lemma~\ref{lem:cat1:nested} the intersection $Y=\bigcap_{i\geq 0}
X_{\alpha_i}$ is a non-empty closed convex subset of $X$. Furthermore by definition we have $d_Y(z_\beta)=\frac
\pi 2$. In particular, we deduce by the same argument as above that $\dim (Y)<n$.

Now for each $\alpha \in A$, we apply Lemma~\ref{lem:cat1:nested} to the nested family $(X_\alpha \cap
X_{\alpha_i})_{i \geq 0}$, which shows that $Y_{\alpha}=\bigcap_{i\geq 0} (X_\alpha \cap X_{\alpha_i})$ is a
closed convex non-empty subset of $Y$ with intrinsic radius at most $\frac \pi 2$. Moreover, the family
$\{Y_{\alpha}\}_{\alpha \in A}$ is filtering and we have $\bigcap_\alpha Y_{\alpha} =\bigcap_\alpha
(X_{\alpha})$. It follows by induction that $\bigcap_\alpha  X_{\alpha}$  is non-empty and of intrinsic radius
at most $\frac \pi 2$, as desired.

\smallskip%
It remains to consider the case when $r = \sup_{\alpha} d_{X_{\alpha}}(z_\beta)<\pi /2$. We are then in a
position to apply Lemma~\ref{lem:cat1:FiniteIntersection} to the filtering family $\{ B(z_\beta, r) \cap
X_{\alpha} \}_{\alpha \in A}$. We deduce that $Y=\bigcap_\alpha X_{\alpha}$ is non-empty. Moreover, since
$d_Y(z_\beta)\leq r < \frac \pi 2$, we deduce by considering the nearest point projection of $z_\beta$ to $Y$
(see \cite[II.2.6(1)]{Bridson-Haefliger}) that $\rad(Y)< \frac{\pi}{2}$.
\end{proof}

\subsection{\cat case}

We start with the special case of nested sequences of convex sets.

As pointed out to us by the referee, the use of the gradient flow in
the argument below is reminiscent of the proof of Lemma~5 on p.~217
in~\cite{BGS95}.

\begin{lem}\label{lem:gradient}
Let $X$ be a complete \cat space of telescopic dimension~$n < \infty$ and $(X_i)_{i \geq 0}$ be a nested
sequence of closed convex subsets such that $\bigcap_{i \geq 0} X_i$ is empty. Let $o \in X$ be a base point and
set $f_i : x \mapsto d_{X_i}(x) - d_{X_i}(o)$.  Then the sequence $(f_i)_{i \geq 0}$ sub-converges to a
$1$-Lipschitz convex function $f$ which satisfies $\inf_{p \in X} |\nabla_p(-f) | \geq \frac{1}{2}\big( 1 -
\sqrt{\frac{n}{n+1}} \big)$.
\end{lem}

\begin{proof}
The functions $f_i$ are $1$-Lipschitz and convex (\cite[II.2.5(1)]{Bridson-Haefliger}), hence they are elements
of the space $\mathcal{C}_0$ defined in Section~\ref{spacefunction}. Since $\mathcal{C}_0$ is compact, the
sequence $(f_i)$ indeed sub-converges to a function $f \in \mathcal{C}_0$. It remains to estimate the absolute
gradient of $f$.

Pick a point $p \in X$. By assumption the intersection $\bigcap_i X_i$ is empty. Since bounded closed convex
sets enjoy the finite intersection property (see Section~\ref{sec:Jung:cat0}), it follows that $d_{X_i}(p)$
tends to infinity with $i$. Thus for each $t >0$ there is some $N_t$ such that $d_{X_i}(p) > t$ for all $i \geq
N_t$. We may and shall assume without loss of generality that $N_1=0$.

Let $x_i$ denote the nearest point projection of $p$ to $X_i$ (see \cite[Proposition~II.2.4]{Bridson-Haefliger})
and $\rho_i : [0, d(p, x_i)] \to X$ be the geodesic path joining $p$ to $x_i$. Set
$$D_t = \diam \{\rho_i(t) \; | \; i \geq N_t\}.$$
We distinguish two cases.

Assume first that $\sup_t D_t < \infty$. It then follows that for all $t >0$, the sequence $(\rho_i(t))_{i \geq
N_t}$ is Cauchy. Denoting by $\rho(t)$ its limit, the map  $\rho : t \mapsto \rho(t)$ is a geodesic ray
emanating from $p$. Therefore $f= \lim_i f_i$ is a Busemann function and we have $|\nabla_p(-f)| = 1$. Thus we
are done in this case.

Assume now that $\sup_t D_t = \infty$. Then $D_t$ tends to infinity with $t$. Choose  $\delta >0$ small enough
so that $\sqrt{2} \delta < 1 - \sqrt{\frac{n}{n+1}}$ and let $D>0$ be the constant provided by
Theorem~\ref{raddiam}. We now pick $t$ large enough so that $D_t > D$ and set $y_i = \rho_i(t)$ for all $i \geq
N_t$. For $j> i$ we have $\angle_{x_i}( p,x_j) \geq \frac {\pi} 2$ and considering a comparison triangle for
$\Delta(p , x_i, x_j)$ we deduce $d(y_i,y_j) \leq t \sqrt 2 $. Set $Y=\{y_i\; | \; i \geq 0\}$. By
\tref{raddiam} we have $\rad(Y) \leq t (\sqrt{2}\delta + \sqrt{\frac{n}{n+1}})$.

Let $z$ be the circumcentre of $Y$. We have $d_{X_i}(z) \leq d_{X_i}(y_i) + d(y_i,z)$ and $d(y_i,z)\leq
\rad(Y)$. Since moreover $d_{X_i}(p) = d_{X_i}(y_i) + t$, we deduce
$$\begin{array}{rcl}
f_i (z) -f_i (p)  &= &d_{X_i}(z) - d_{X_i}(p) \\
& \leq & d(y_i,z) - t\\
& \leq & - t (1- \sqrt{\frac{n}{n+1}} -\sqrt{2}\delta)
\end{array}
$$
for each $i \geq 0$. Therefore $f(p)-f(z) \geq \delta' t$, where $\delta' = 1- \sqrt{\frac{n}{n+1}}
-\sqrt{2}\delta$.  On the other hand, we have $d(z,p) \leq d(p,y_i)+d(y_i,z) \leq 2t$, thus
$\frac{f(p)-f(z)}{d(p, z)} \geq \delta'/2$.

Since the restriction of $-f$ to the geodesic segment $[p, z]$ is concave by assumption, we deduce
$$|\nabla_p(- f)| \geq \delta' /2.$$
Finally, recalling that $\delta'= 1- \sqrt{\frac{n}{n+1}} -\sqrt{2}\delta$ and that $\delta>0$ may be chosen
arbitrary small, the desired estimate follows.
\end{proof}

\begin{lem}\label{lem:nested}
Let $X$ be a complete \cat space of finite telescopic dimension and $(X_i)_{i \geq 0}$ be a nested sequence of
closed convex subsets. If $\bigcap_{i \geq 0} X_i$ is empty, then $\bigcap_{i \geq 0} \bd X_i$ is a non-empty
subset of the visual boundary $\bd X$ of intrinsic radius~$\leq \frac{\pi}{2}$.
\end{lem}

\begin{proof}
Let $\phi _t: X \to X$ denote the gradient flow associated to the convex function $f$ defined as in
Lemma~\ref{lem:gradient}. \pref{karlsson} provides some point $\xi$ in the ideal boundary $\partial X$ such that
the gradient line $t \mapsto \phi _t (p)$ converges to $\xi$  for any starting point $p\in X$.

We claim that $\xi$ is contained in $\partial X_i$ for each $i$. To this end, we fix an index $i$ and consider
the restriction $h$ of $f$ to $X_i$. This is a convex function on $X_i$  and it is sufficient to prove that the
gradient  flow of $h$ coincides with the gradient flow of $f$ starting at any  point of $X_i$.  Hence it is
enough to prove that for all $p \in X_i$ the equality $|\nabla _p(-f)|=|\nabla _p(-h)|$ holds.

Pick a point $x \in X$ and let $x_i$ denote the nearest point projection of $x$ to $X_i$. We have $ d_{X_j}(x)
\geq d_{X_j}(x_i) $ and $d(p,x)\geq d(p, x_i)$ for all $p \in X_i$. Hence for $p \in X_i$ and all $j\geq i$ we
get the inequality
$$\frac {f _j(p) - f_j(x)}{d(p, x)} \leq   \frac {f_j (p) - f_j (x_i)}{d(p, x_i)}.$$
Hence the same is true for the limiting function $f$, which implies the desired equality $|\nabla _p
(-f)|=|\nabla _p(-h)|$. This shows that $\xi$ is contained in the intersection $\bigcap_i \partial X_i$, which
is thus non-empty.

\smallskip%
For any geodesic ray $\eta$ in $X$ with endpoint in $\bigcap_i \partial X_i$, the restriction of $f_i $  to
$\eta$ is bounded from above, hence non-increasing. Therefore the same holds true for the restriction of the
limiting function $f$ to the ray $\eta$. In other words the endpoint of $\eta$ is $f$-monotone.  From
\lref{eberle} we deduce that $d(\xi, \psi ) \leq \pi /2$ for all $\psi \in  \cap \partial X_i$.
\end{proof}

\begin{proof}[Proof of \tref{filtering}]
Pick a base point $o \in X$. If the set $\{d_{X_\alpha}(o)\}_{\alpha \in A}$ is bounded, then $\bigcap_\alpha
X_\alpha$ has a non-empty intersection by the finite intersection property (see Section~\ref{sec:Jung:cat0}). We
assume henceforth that this is not the case. In particular there exists a sequence of indices $(\alpha_n)_{n
\geq 0}$ such that $\lim_n d_{X_{\alpha_n}}(o) = \infty$. Now for each $\alpha \in A$, we may apply
Lemma~\ref{lem:nested} to the nested sequence $(X_\alpha \cap X_{\alpha_n})_{n \geq 0}$. This shows that
$Y_\alpha = \bigcap_{n \geq 0} \bd (X_\alpha \cap X_{\alpha_n})$ is a non-empty subset of intrinsic radius~$\leq
\frac{\pi}{2}$ of $\bd X$. Notice that $\{Y_\alpha\}_{\alpha \in A}$ is a filtering  family.
Proposition~\ref{dimension} then allows one to appeal to Proposition~\ref{prop:cat1:filtering}, which shows that
$\bigcap_\alpha Y_\alpha$ is a non-empty subset of intrinsic radius~$\leq \frac{\pi}{2}$. This provides the
desired statement since $\bigcap_\alpha \bd X_\alpha = \bigcap_\alpha Y_\alpha$.
\end{proof}

We end this section by an example illustrating that Theorem~\ref{filtering} fails if one assumes only that the
Tits boundary $\bd X$ be finite-dimensional.

\begin{ex}\label{ex:Hilbert}
Let $\mathcal{H}$ be a separable (real) Hilbert space with
orthonormal basis $\{e_i\}$ and $X \subset \mathcal{H}$ be the
subset consisting of all points $\sum_i a_i e_i$ with $|a_i|\leq i$
for all $i$. Thus $X$ is a closed convex subset of $\mathcal{H}$
with \emph{empty} (hence finite-dimensional) ideal boundary. Let now
$X_n = \{\sum_i a_i e_i \in X \; | \; a_i \geq 1 \text{ for all }
i\leq n\}$. Then $\{X_n\}$ is a nested family of closed convex
subsets with empty intersection.
\end{ex}

\section{Applications}

\subsection{Parabolic isometries}

\begin{proof}[Proof of \cref{FixedPointParabolic}]
By Proposition~\ref{dimension}, the boundary $\bd X$ is finite-dimensional. The sublevel sets of the
displacement function of $g$ define a $\centra_{\Isom(X)}(g)$-invariant nested sequence of closed convex
subspace. The intersection of their boundaries is nonempty by Theorem~\ref{filtering} and possesses a barycentre
by \cite[Prop.~1.4]{BalserLytchak_Centers}, which is the desired fixed point.
\end{proof}

\subsection{Minimal and reduced actions}

We begin with a de Rham type decomposition property. It was shown by Foertsch--Lytchak \cite{FoertschLytchak06}
that any finite-dimensional \cat space (and more generally any geodesic metric space of finite affine rank)
admits a canonical isometric splitting into a flat factor and finitely many non-flat irreducible factors.
Building upon \cite{FoertschLytchak06}, it was then shown by Caprace--Monod
\cite[Corollary~4.3(ii)]{CapraceMonod} that the same conclusion holds for proper \cat spaces whose isometry
group acts minimally, assuming that the Tits boundary is finite-dimensional. We shall need the following
`improper' variation of this result.

\begin{prop}\label{FoertschLytchak}
Let $X$ be a complete \cat space of finite telescopic dimension, such that $\Isom(X)$ acts minimally. Then there
is a canonical maximal isometric splitting
$$\RR^n \times X_1 \times \cdots \times X_m$$
where each $X_i$ is irreducible, unbounded and~$\not \cong \RR$. Every isometry preserves this decomposition
upon permuting possibly isometric factors $X_i$.
\end{prop}

\begin{proof}
Let $\mathcal{H}$ be a separable Hilbert space with orthonormal basis $\{e_i\}_{i >0}$ and denote by $C_k$ the
convex hull of the set $\{0\} \cup \{2^i e_i \; | \; 0< i \leq  k\}$. Let now $X$ be a \cat space such that for
every isometric splitting $X = X_1 \times \cdots  \times X_p$ with each $X_i$ unbounded, some factor $X_i$
admits an isometric splitting $X_i = X_i' \times X_i''$ with unbounded factors. Then there is a point $o \in X$
and for each $k >0$ an isometric embedding $\varphi_k : C_k \to X$ with $\varphi(0) = o$. Since for all $k>0$
the set $2.C_k$ embeds isometrically in $C_{k+1}$, it follows that $C_k$ embeds isometrically in the asymptotic
cone $\lim_\omega (\frac{1}{n}X, o)$. In particular $X$ does not have finite telescopic dimension. This shows
that any \cat space of finite telescopic dimension admits a \emph{maximal} isometric splitting into a product of
finitely many unbounded (necessarily irreducible) subspaces.

In view of the latter observation and given Proposition~\ref{dimension}, the proof of
\cite[Corollary~4.3(ii)]{CapraceMonod} applies \emph{verbatim} and yields the desired conclusion.
\end{proof}

\begin{proof}[Proof of \pref{MinimalReduced}]
(i)  We claim that the statement of (i) follows from (ii) and (iii). Indeed, if $G$ has no fixed point at
infinity, then there is a minimal non-empty $G$-invariant subspace $Y \subseteq X$ by (ii). Upon replacing $G$
by a finite index subgroup, this subspace $Y$ admits a $G$-equivariant decomposition as in
Proposition~\ref{FoertschLytchak}. The induced action of $G$ on each of these spaces is minimal without fixed
point at infinity. Therefore, it is non-evanescent by (iii), unless $Y$ is bounded, in which case it is reduced
to a single point by $G$-minimality. This means that $G$ fixes a point in $X$.

\smallskip \noindent%
(ii) Assume that $G$ has no minimal invariant subspace. By Zorn's lemma this implies that there is a chain of
$G$-invariant subspaces with empty intersection. By Theorem~\ref{filtering} the intersection of the boundaries
at infinity of the subspaces in this chain provide a closed convex $G$-invariant set $Y \subseteq \bd X$ of
radius $\leq \frac \pi 2$. By Proposition~\ref{dimension}, the set $Y$ is finite-dimensional. Hence it possesses
a unique barycentre by \cite[Prop.~1.4]{BalserLytchak_Centers}, which is thus fixed by $G$.

\smallskip %
By Proposition~\ref{dimension} the boundary $\bd X$ is finite-dimensional. Therefore, for (iii) and (iv),
Theorem~\ref{filtering} (in fact, Lemma~\ref{lem:nested} is sufficient) allows one to repeat \emph{verbatim} the
proofs of the corresponding statements that are given in \cite{CapraceMonod}, namely Theorem~1.6 in
\emph{loc.~cit.} for the fact that normal subgroups act minimally without fixed point at infinity, Corollary~2.8
in \emph{loc.~cit.} for the fact that the $G$-action is reduced and Proposition~1.3(i) in \emph{loc.~cit.} for
the fact that $X$ is boundary-minimal provided $\Isom(X)$ acts minimally.
\end{proof}

\begin{proof}[Proof of \cref{superrigidity}]
By Proposition~\ref{FoertschLytchak} the space $X$ admits a canonical  decomposition as a product of finitely
many irreducible factors. The lattice $\Gamma$ admits a finite index normal subgroup $\Gamma^*$ which acts
componentwise on this decomposition (the  finite quotient $\Gamma/\Gamma^*$ acts by permuting possibly isometric
irreducible factors). Let $G_i^*$ be the closure of the projection of $\Gamma^*$ to $G_i$ and set $G^*= G_1^*
\times \cdots \times G^*_n$. Thus $G^*$ is a closed normal subgroup of finite index of $G$ and we have $G =
\Gamma \cdot G^*$. In particular it is sufficient to show that the $\Gamma^*$-action extends to a continuous
$G^*$-action. To this end, we work one irreducible factor at a time. Given
Proposition~\ref{MinimalReduced}(iii), the desired continuous extension is provided by
\cite[Theorem~6]{Monod_superrigid}.
\end{proof}

\subsection{Isometric actions of amenable groups}

\begin{proof}[Proof of \tref{AdamsBallmann}]
Assume that $G$ has no  fixed point at infinity. Thus there is a minimal closed  convex invariant subset by
Proposition~\ref{MinimalReduced}(ii) and we may assume that this subset coincides with $X$. In other words $G$
acts minimally on $X$. Let $X= \mathbb{E} \times X'$ be the maximal Euclidean decomposition (see
\cite[II.6.15(6)]{Bridson-Haefliger}). Thus $G$ preserves the splitting  $X= \mathbb{E} \times X'$ and the
induced $G$-action on both $\mathbb{E}$ and $X'$ is minimal and does not fix any point at infinity. We need to
show that $X'$ is reduced to a single point. To this end, it is thus sufficient to establish the following
claim.

\emph{If an amenable locally compact group $G$ acts continuously, minimally  and without fixed points at
infinity on a \cat space $X$ of finite telescopic dimension without Euclidean factor, then $X$ is reduced to a
single point.}

\smallskip%
Assume that this is not the case. Pick a base point $o \in X$ and consider the spaces $\mathcal{C} \subset
\mathcal{C}_0$ defined in Subsection~\ref{spacefunction}.  Let $\mathcal A$ denote the closed convex hull of
$\overline{ \mathcal{C}}$ in the locally convex topological vector space $\mathcal B$ of all functions vanishing
at $o$. By Proposition~\ref{FoertschLytchak} the subset $\overline{\mathcal{C}}$ does not contain any affine
function. It follows from Lemma~\ref{nonaffine} that $\mathcal{A}$ does not contain any affine function either.
The induced action of $G$ on $\mathcal B$ is continuous and preserves the compact convex set $\mathcal A$. By
the definition of amenability $G$ has a fixed point in $\mathcal A$. Thus we have found some \emph{non-constant}
$1$-Lipschitz convex function $f$ which is \textbf{quasi-invariant} with respect to $G$ in the sense that, for
each $g \in G $, one has  $f(gx)= f(x)+ f(g o)$. (In other words, this means that for each $g$, the map $x
\mapsto f(gx)-f(x)$ is constant.) The following lemma, analogous to \cite[Lemma~2.4]{Adams-Ballmann}, implies
that $G$ has a fixed point at infinity, which is absurd.
\end{proof}

\begin{lem}
Let a group $G$ act minimally by isometries on a complete \cat space $X$ of finite telescopic dimension. There
is a $G$-quasi-invariant continuous non-constant convex function $f$ on $X$ if and only if $G$ fixes a point in
$\partial X$.
\end{lem}

\begin{proof}
If $G$ fixes a point in $\partial X$, then the Busemann function of this point (that is uniquely defined up to a
positive constant) is quasi-invariant.

Assume that $f$ is quasi-invariant and define $a:G\to \mathbb R$ by $a(g)=f(gx)-f(x)$. By assumption $a$ does
not depend on  $x$; furthermore $a$ is a homomorphism. If $a$ were constant, then $f$ would be $G$-invariant
and, hence, so would be any sub-level set of $f$. This contradicts the minimality assumption on the $G$-action.
Therefore $a$ is non-constant; more precisely the image of $a$ is unbounded and $\inf f=-\infty$. For each $r
\in \RR$ set $X_r :=\phi^{-1} (-\infty, -r]$. Then $(X_r)_{r \in R}$ is a chain of closed convex subspaces with
empty intersection; furthermore every element of $G$ permutes the sets $X_r$. It follows that $C=\bigcap_{r \in
\RR} \partial X_r $ is $G$-invariant. Theorem~\ref{filtering} now shows that $C$ is nonempty of radius $\leq
\frac{\pi}{2}$, and \cite[Prop.~1.4]{BalserLytchak_Centers} implies that $G$ fixes a point in $C \subset \bd X$.
\end{proof}

\begin{proof}[Proof of Theorem~\ref{cell}]
The proof mimicks the arguments given in \cite{CapraceTD}; we do not reproduce all the details. As in
\emph{loc.~cit.} the key point is to establish that every point of the \textbf{refined boundary} $\bdfine X$
(defined in \emph{loc.~cit.}, \S4.2) has an amenable stabiliser in $G$ and that, conversely, any amenable
subgroup of $G$ possesses a finite index subgroup which fixes a point in $X \cup \bdfine X$. The proof that
amenable groups stabilise point in $X \cup \bdfine X$ uses \tref{AdamsBallmann} together with an induction on
the geometric dimension (see the remark following Corollary~4.4 in \emph{loc.~cit.} showing that there is a
uniform upper-bound on the level of a point in the refined boundary). For the converse, one shows directly that
the $G$-stabiliser of a point in $\bdfine X$ is (topologically locally finite)-by-(virtually Abelian); the
cocompactness argument used in Proposition~4.5 of \emph{loc.~cit.} is replaced by a compactness argument relying
on the hypothesis that $X$ has finitely many types of cells, all of which are compact.
\end{proof}

\begin{proof}[Proof of \cref{radical}]
By~Proposition~\ref{FoertschLytchak}, the space $X$ admits a canonical  decomposition as a product of a maximal
Euclidean factor and a finite number of irreducible non-Euclidean factors. The Euclidean factor is $G$-invariant
and $G$ possesses a closed normal subgroup of finite index $G^*$ that acts componentwise on the above product.
By hypothesis, the $G^*$-action on each non-Euclidean factor is minimal and does not fix any point at infinity.
\tref{AdamsBallmann} and Proposition~\ref{MinimalReduced}(iii) therefore imply that the amenable radical of
$G^*$ acts trivially. This implies that the amenable radical of $G$ acts as a finite group on the product of all
non-Euclidean factors of $X$. Thus this action is trivial since $G$ acts minimally.
\end{proof}

\newpage

\appendix

\pagestyle{empty}

\begin{center}
\textbf{\large{Erratum to\\ `At infinity of finite-dimensional CAT(0) spaces'}}

\medskip
 Pierre-Emmanuel Caprace and Alexander Lytchak

\medskip
First draft December 2012; revised May 2014
\end{center}

\bigskip

A gap in the proofs of Propositions~1.8 and~6.1 from our paper~\cite{CL} was pointed out to us by Pierre Py: we overlooked the possibility that a complete \cat space $X$ of finite telescopic dimension could be unbounded, have an isometry group acting minimally, and nevertheless have an empty visual boundary. Although it is still not clear to us whether this situation can actually occur\footnote{This question has recently been answered negatively by Bader--Duchesne--L\'ecureux in \cite[Th.~1.2]{BDL}. This fills in the gap from \cite{CL} in an optimal way, thereby making the arguments developed in the present erratum obsolete. This note is therefore not intended for publication.}, we show in this note how   to overcome the question and complete the proofs which were flawed in the original paper (only the proofs of Prop.~1.8 and~6.1, and of Corollaries 1.9 and 1.10, which were all given in the final section of the original paper~\cite{CL}, are concerned by these adjustments). Besides Proposition~1.8 which has to be slightly corrected (see Proposition~1 below), all the other statements of the paper are correct and remain unchanged.

We thank Pierre Py for pointing out the gap. We are also grateful to him for letting us know  of the unpublished manuscript \cite{KS99} by Korevaar--Schoen, which we were not aware of while writing our paper. In  \cite{KS99}, a class of \cat spaces called \textbf{FR-spaces}, is introduced. It turns out that a \cat space is an FR-space in the sense of   Korevaar--Shoen if and only if it has finite telescopic dimension in the sense of our paper~\cite{CL}: this equivalence can be deduced from Theorem~1.3 from~\cite{CL}.  Given this equivalence, one sees that Proposition~1 from  \cite{KS99} overlaps somewhat Theorem~1.1 from~\cite{CL}. The condition of finite telescopic dimension seems however somewhat easier to check in concrete settings; for example it was used by B.~Duchesne \cite{Duc}, in combination with the aforementioned equivalence, to show that infinite-dimensional symmetric spaces of finite rank are FR-spaces, thereby answering a question asked by Korevaar--Shoen in  \cite{KS99}. 

\medskip
By convention, all the numbered references of the form `Theorem~1.1'  in this note refer to the corresponding results from \cite{CL}.  We moreover retain the terminology used in loc.~cit. Proposition 1.8 must be replaced by the following.

\begin{propapp}\label{app:1.8}
Let $X$ be a complete \cat space of finite telescopic dimension and $G < \Isom(X)$ be any group of isometries. 

\begin{enumerate}[(i)]
\item If the $G$-action is minimal, then it is evanescent if and only if it fixes a point in $ \bd X$. 

\item If $G$ does not fix a point in $\bd X$, then there is a non-empty $G$-invariant closed convex subset $Y \subseteq X$ on which $G$ acts minimally. 

\item Assume that $X$ is irreducible, not isometric to the real line $\mathbf R$. If $G$ acts minimally without a fixed point at infinity, then so does every non-trivial normal subgroup. Moreover, for any unbounded closed convex subset $T$ which is mapped at bounded Hausdorff distance from itself by each element of $G$, we must have $X=T$  unless    $\bd T = \varnothing$ and  $T$ is not $\Isom(T)$-minimal. 

\item If $\Isom(X)$ acts minimally, then $X$ admits a unique product decomposition $X = X' \times T$ where $X'$ is boundary-minimal and $T$ has empty boundary. 
\end{enumerate}
\end{propapp}

A closed convex subset $Y$ of a \cat space $X$ is called \textbf{boundary-minimal} if $\bd Z \subsetneq \bd Y$ for every closed convex subset $Z \subsetneq Y$. 

The proof of Proposition~\ref{app:1.8} requires some preparation.

\subsection*{Ultracompletions}
We shall use the \textbf{ultracompletion} of a \cat space $X$, with respect to an ultrafilter $\omega$. It is denoted by $X^\omega$ and defined as the ultralimit  $\lim _{\omega} (X,o)$ of the constant sequence of pointed metric spaces $(X, o)$.  {The ultracompletion is also called \textbf{ultraproduct} or \textbf{ultrapower} by some authors; we prefer the term \emph{ultracompletion}, notably because of its basic properties, collected in the following, which strongly suggest that the space $X^\omega$ is some kind of completion of $X$.} 

\begin{lemapp}\label{app:lem:ultra}
Let $X$ be a \cat space. 
\begin{enumerate}[(i)]
\item $X^{\omega}$ is a complete \cat space in which $X$ embeds  canonically as a convex  subspace. Moreover any isometric action of a group $G$ on $X$ extends to an isometric action of $G$ on $X^{\omega}$.

\item  $X$ is unbounded if and only if $\partial (X^{\omega} ) \neq \varnothing$.

\item If $\omega _1$ and $\omega _2$ are ultrafilters then the iterated ultracompletion  $(X^{\omega _1}) ^{\omega _2}$
is (canonically) isometric to an ultracompletion of $X$ with respect to the ultrafilter $\omega _1 \times \omega_2$.

\item  The telescopic dimensions of $X$ and $X^{\omega}$ coincide.

\item If $X$ has a finite telescopic dimension $d$, then for any ultrafilter $\omega$, the geometric dimension of
$\partial (X^{\omega})$ is at most $d-1$. 
\end{enumerate}
\end{lemapp}
 
\begin{proof}
For (i) and (ii), see Cor.~II.3.10 in \cite{BH}.  {For (iii) see \cite[\S 3.2]{DrutuSapir}.} (iv) is a consequence of Theorem~1.3. Taking ultralimits of \cat spaces does not increase the geometric dimension by Lemma~11.1 from~\cite{Lytchak}; therefore (v) follows from Prop.~2.1. 
\end{proof}

 \medskip

\begin{lemapp} \label{app:nonprod}
Let $X$ be a \cat space of finite telescopic dimension.  Then $X^{\omega}$ does not contain any copy of $X\times I$ with $I$ an 
interval  of positive length.
\end{lemapp}

\begin{proof}
Otherwise, an iterated ultracompletion contains $X\times I^n$. Notice that the $n$-dimensional Euclidean cube of edge length $\sqrt 2 c$ embeds isometrically in the $2n$-dimensional Euclidean cube of edge length $c$. It follows that for $n$ large, the space $X \times I^n$ contains an arbitrary large Euclidean
ball of arbitrary large dimension.  But an iterated ultracompletion is an ultracompletion for another ultrafilter
 $\omega '$ by Lemma~\ref{app:lem:ultra}(iii).  Therefore $X$ cannot have finite telescopic dimension {(by Theorem 1.3)}, a contradiction.
\end{proof}

\begin{lemapp} \label{app:ultra}
Let $G < \Isom(X) $ act minimally on a \cat space $X$ of finite telescopic dimension, with no fixed points in $\partial X$.  Then the extended  action of $G$ on $X^{\omega}$  does not have fixed points in the (larger) boundary $\partial (X^{\omega})$. 
\end{lemapp}

\begin{proof}
Otherwise, the corresponding Busemann function {on $X^\omega$} is quasi-invariant under $G$. By  Lemma~6.2,
it must be constant on $X$. It follows that the product  $X\times [0,\infty )$ embeds isometrically in $X^{\omega}$, contradicting \lref{app:nonprod}.
\end{proof}

\begin{lemapp} \label{app:contains}
Let $G  < \Isom(X)$ act minimally on a complete \cat space $X$ of finite telescopic dimension.  Then any $G$-invariant closed convex  subset of $X^{\omega}$
contains (the canonical) subset $X$ of $X^{\omega}$.
 \end{lemapp}

\begin{proof}
Let $\mathscr C$ be the collection of all $G$-invariant closed convex subspaces of $X^{\omega}$ on which $G$-acts minimally. Then $\mathscr C$ is non-empty since it contains $X$. An easy argument using the Sandwich Lemma (\cite[II.2.12(2)]{BH}) then shows that every $G$-invariant closed convex subspace contains a minimal one, which is thus an element of $\mathscr C$. Moreover, the union $\bigcup \mathscr C$ is a closed convex subspace which splits as a product of the form $X \times T$ (see \cite[Rem.~39]{Monod}). By \lref{app:nonprod}, the cross-section $T$ must be reduced to a single point.  Hence  $\mathscr C = \{X\}$. 
\end{proof}

\subsection*{Evanescent actions}  

\begin{lemapp}\label{app:lem:eva}
 Let  $G  < \Isom(X)$ act minimally on a \cat space $X$ of finite telescopic dimension. 
The $G$-action  is evanescent if and only if $G$ fixes a point in $\partial X$.
\end{lemapp}

\begin{proof}
The `if' direction is clear. For the reverse implication, consider an unbounded  subset $T\subset X$, on which all elements of $G$
have finite displacement functions. Since the displacement functions of isometries are convex, they remain finite on the convex hull of $T$, and we may thus assume that $T$  is convex. Then $G$ fixes all points in $\partial (T^{\omega} ) \subset
\partial X^{\omega}$, which contradicts \lref{app:ultra} since $\partial (T^{\omega} )$ is non-empty by Lemma~\ref{app:lem:ultra}(ii).
\end{proof}

\subsection*{Minimal actions and boundary-minimality}

\begin{lemapp}\label{app:lem:BdMin}
Let $X$ be a complete \cat space of finite telescopic dimension and $Y \subseteq X$ be a closed convex subspace. Assume that the boundary $\bd Y$ has radius~$> \pi/2$. Then $X$ contains some boundary-minimal closed convex subspace $Y'$ with $\bd Y' = \bd Y$. Moreover, the union of all such subspaces splits as a product $Y' \times T$, and  if $\bd Y = \bd X$ then $T$ has empty visual boundary.  
\end{lemapp}

\begin{proof}
Any chain of closed convex subspaces with visual boundary equal to $\bd Y$ has a non-empty intersection, since otherwise $\bd Y$ would have radius~$\leq \pi/2$ by Theorem~1.1. Thus the existence of $Y'$ follows from Zorn's lemma. That the union of all boundary-minimal subspaces with a given boundary   is a product space is true in arbitrary \cat spaces, and can be shown as in \cite[Prop.~3.6]{CM}. 
\end{proof}

\begin{lemapp}\label{app:lem:BdMin2}
Let $X$ be a complete \cat space of finite telescopic dimension, such that $\Isom(X)$   acts minimally. Then either  $\bd X$ has radius~$>\pi/2$, or $\bd X$ is empty.
\end{lemapp}

\begin{proof}
Assume for a contradiction that $\bd X$ is non-empty of radius~$\leq \pi/2$. By Proposition 2.1 and \cite[Prop.~1.4]{BalserLytchak}, the group $G= \Isom(X)$ fixes some point $\xi \in \bd X$, and the ball of radius $\pi/2$ around $\xi$ is the entire boundary. As in  \cite[Prop.~3.11]{CM}, one then shows that $G$ preserves all horoballs centered at $\xi$, contradicting minimality.  
\end{proof}

\begin{lemapp}   \label{app:strange}
Let $X$ be a complete \cat space of finite telescopic dimension {not reduced to a single point}, such that $\Isom(X)$   acts minimally.  Then $\bd(X^{\omega})$ has radius~$>\pi/2$. 

If in addition the ultrafilter $\omega$ is chosen so that the dimension of $ \partial (X^{\omega})$ is maximal (such a choice of $\omega$ exists by Lemma~\ref{app:lem:ultra}(v)), then      there is a unique boundary-minimal
 subspace $X'$ of $X^{\omega}$ with $\bd X' = \bd (X^{\omega})$. Moreover $X'$ contains $X$.
\end{lemapp}

\begin{proof}
If $X$ is bounded, then it is a point by minimality, {which is excluded by hypothesis}. We assume henceforth that $X$ is unbounded, hence  $\bd (X^{\omega})$ is non-empty and of finite geometric dimension by Lemma~\ref{app:lem:ultra}(ii) and (v).

Assume for a contradiction that the radius of  $\bd(X^{\omega})$ is smaller than or equal to $\pi/2$. Then $G = \Isom(X)$ fixes some circumcentre  $\xi$ in $\bd(X^{\omega})$ by \cite[Prop.~1.4]{BalserLytchak}; {the whole boundary $\bd(X^{\omega})$ is thus contained in the ball of radius $\pi/2$ around $\xi$. By \cite[Lem.~3.12]{CM}, this implies that $G$ stabilizes each horoball around $\xi$. We deduce from Lemma~\ref{app:contains} that $X \subset X^\omega$ is contained in the intersection of all these horoballs, which is absurd.   Hence $\bd(X^{\omega})$ has radius~$>\pi/2$.}

By Lemma~\ref{app:lem:BdMin} there is a boundary minimal subspace $X'$ of $X^{\omega}$ with full boundary, and the union   of all such   subspaces   has the form $Y = X' \times T$, by Lemma~\ref{app:lem:BdMin}. 

We now assume that $\omega$ is chosen so that $\dim( \partial (X^{\omega}))$ is maximal.
If $T$ is unbounded, then the ideal boundary of any  ultracompletion of $Y$  contains the spherical join of 
$\partial X' = \partial (X^\omega)$
with a singleton. Such a subspace has larger dimension than $\partial X^{\omega}$, which   contradicts the choice of  $\omega$.

Therefore,  the cross-section $T$ is bounded. Notice that the $G$-action  on $X^{\omega}$ preserves $Y$ and its product decomposition.
Since $T$ is bounded, the induced action of $G$ on $T$ has a fixed point. Thus $G$ preserves some fiber along $X'$ in the product $X' \times T$.
By \lref{app:contains} this fiber must contain $X$, hence  $T$ must be reduced to a single point by \lref{app:nonprod}.
Thus  $Y = X'$ and  $X\subset X'$. 
\end{proof}
 
 \begin{lemapp}\label{app:lem:1.8(iv)}
Let $X$ be a complete \cat space of finite telescopic dimension. If $\Isom(X)$   acts minimally, then $X$ admits a unique product decomposition $X = X' \times T$ where $X'$ is boundary-minimal and $T$ has empty boundary. 
\end{lemapp}

\begin{proof}
If $\bd X$ is non-empty, then it has radius~$>\pi/2$ by Lemma~\ref{app:lem:BdMin2}. We may then invoke Lemma~\ref{app:lem:BdMin}, which provides a canonical product decomposition $X = X' \times T$, where $X'$ is boundary-minimal and $T$ is unbounded and has empty boundary. 

If $\bd X$ is empty, the desired decomposition holds trivially by setting $T= X$ and $X'$ a singleton. 
\end{proof}

\subsection*{Reduced actions}

Following N.~Monod \cite{Monod}, the action of a group $G$ on a \cat space $X$ is called \textbf{reduced} if there is no unbounded closed convex subset $T \subsetneq X$ which is mapped at bounded Hausdorff distance from itself by every element of $G$. 

\begin{lemapp}\label{app:lem:reduced}
Let $G   < \Isom(X)$ act minimally on a  complete \cat space $X$ of finite telescopic dimension, without a fixed point at infinity.
Assume that $X$ is irreducible, and let $T$ be an unbounded closed convex subset, such that
$d(gT,T) <\infty$, for all $g\in G$.  If $\partial T$ is non-empty, or if  $\Isom(T)$ acts minimally on $T$, then $T =X$.
\end{lemapp}

\begin{proof}
If the boundary of $T$ is non-empty, our old proof applies: $\bd T$ is a $G$-invariant closed convex subset of the boundary, and must be of radius~$>\pi/2$ by \cite[Prop.~1.4]{BalserLytchak}. Lemma~\ref{app:lem:BdMin} then shows that the (non-empty) union of all boundary-minimal closed convex subsets with boundary $\bd T$ is a product space, which is $G$-invariant. Since $G$ acts minimally and $X$ is irreducible, it follows that $X$ is boundary-minimal and that $\bd X = \bd T$. In particular $X=T$. 

Assume now that $T$ is $\Isom(T)$-minimal. Every element of $G$ maps the subset $T^{\omega}$ of $X^{\omega}$  at bounded Hausdorff distance from itself.
Hence any element of $G$ preserves the boundary $S =\partial T^{\omega}$, which is non-empty by Lemma~\ref{app:lem:ultra}(ii) and has radius $> \pi /2$ by  \cite[Prop.~1.4]{BalserLytchak} and Lemma~\ref{app:ultra}. We are thus in a position to apply Lemma~\ref{app:lem:BdMin} to the subspace $T^\omega \subseteq X^\omega$. This shows that the (non-empty) union of all boundary-minimal closed convex subsets with boundary $S$ is a product space, which is $G$-invariant. Thus this union has the form $T' \times T''$, where $T'$ is some boundary-minimal subset of $T^{\omega}$. By  \lref{app:strange}, and up to replacing $\omega$ by an ultrafilter making $\dim(\bd T^\omega)$ maximal, the hypothesis that $T$ is $\Isom(T)$-minimal implies that  $T'$ is unique. More importantly, it contains $T$.

Remark  that for any \cat space $X$ and any convex subspace $T$,  the nearest point projection
of $X$ to $T^{\omega}$ inside $X^{\omega}$ has $T$ as its image, due to the very definition of distances in ultracompletions. Thus we have 
$$T \subseteq T' \cap X \subseteq \pi_{T'}(X) \subseteq \pi_{T^\omega}(X) = T,$$
where $\pi_Z$ denotes the nearest point projection to $Z$. It follows that $\pi_{T'}(X) = T$. 

On the other hand, the space  $X$ is contained in the $G$-invariant space $T' \times T''$, whose product decomposition is $G$-invariant. Moreover, on all points of $T' \times T''$, the projection to the first factor $T' \times T'' \to T'$ coincides with the nearest point projection $\pi_{T'}$. We deduce from the preceding paragraph that $X$ is in fact contained in $T \times T''$.  

Let finally $Z'' \subset T''$ be the set of those $z \in T''$ such that the $T$-fiber through $z$ in the product $T \times T''$ is entirely contained in $X$. Since $T \subset X$, the set $Z''$ is non-empty. Moreover $Z''$ is convex (by the Sandwich Lemma) and closed (because $X$ is complete). By construction the product $T \times Z''$ is then a $G$-invariant closed convex subset contained in $X$. By minimality of the $G$-action it coincides with $X$, hence $T= X$ by irreducibility.
\end{proof}

\subsection*{Proof of Proposition~\ref{app:1.8}}

(i) is proved in Lemma~\ref{app:lem:eva}. (ii) follows from Theorem 1.1 and \cite[Prop~1.4]{BalserLytchak}, as in the original argument. (iv) was proved in Lemma~\ref{app:lem:1.8(iv)}. The part of assertion (iii) concerning reduced actions was proved in Lemma~\ref{app:lem:reduced}. It remains to prove the statement on normal subgroups. Given a normal subgroup $N< G$, there are two possibilities: either there is  some non-empty closed convex subset  $Y \subseteq X$ which is $N$-invariant and on which $N$ acts minimally, or there is no such, which implies by (ii) that $N$ fixes a point in $\bd Y$ for  any $N$-invariant closed convex subset $Y \subseteq X$ by (ii). In particular   $\bd Y$ is non-empty. Given this observation, one can proceed with the same argument as in the original proof.
\qed

\subsection*{de Rham decompositions}
In order to fill in  the gap in the proof of  Proposition~6.1, we first record an elementary fact.

\begin{lemapp} \label{app:prodsum}
The  telescopic dimension of a product of two spaces is the sum of the corresponding telescopic dimensions.
\end{lemapp}

\begin{proof}
The statement is true for the geometric dimension.  The product decomposition is stable under 
rescalings and ultralimits.
\end{proof}

The following statement is Proposition 6.1, which we reproduce here for the reader's convenience. 

\begin{propapp}
Let $X$ be a complete \cat space of finite telescopic dimension, such that  $\Isom(X)$  acts minimally on $X$.
Then $X$ has a canonical finite product decomposition $X=\mathbf R^n \times X_1 \times ...\times X_m $, where each $X_i$ is irreducible, unbounded and not isometric to $\mathbf R$.
\end{propapp} 

\begin{proof}
By \lref{app:lem:1.8(iv)}, it is therefore sufficient to prove the proposition for boundary-minimal spaces, and for spaces with an empty boundary. In the former case, we may proceed as in the original argument and get the conclusion. In order to treat the latter case, we may assume that $X$ is unbounded and has empty visual boundary. 

We claim that $X$ does not have  any  bounded  non-trivial factor. Assume the contrary and write $X=X_0 \times C$ with a bounded space $C$.
Then $\partial (X_0 ^{\omega}) =\partial (X ^{\omega})$. Thus $X_0 ^{\omega}$ contains a boundary-minimal subset $X''$ of
$X^\omega$ with full boundary.  For different points $c_1,c_2 \in C$ we obtain disjoint boundary-minimal subsets $X'' \times \{c_i \}$ of $X^{\omega}$, a contradiction.  This proves the claim. In particular any non-trivial factor of $X$ has strictly positive telescopic dimension. 

Due to \lref{app:prodsum}, any product decomposition of $X$ has at most 
$k$ non-trivial factors, where $k$ is the telescopic dimension of $X$.  Thus there is some finite  decomposition of $X$ with irreducible non-trivial factors, which are all unbounded. 
It remains to prove that such a decomposition is canonical.

By Lemma~\ref{app:lem:ultra}(ii) the space $X^{\omega}$ has non-empty visual boundary, and \lref{app:strange} implies (upon changing the ultrafilter $\omega$) that there is a unique boundary-minimal subspace $X'$  of $X^{\omega}$.

The boundary-minimal subset $X'$   of $X^{\omega}$ admits a canonical product decomposition by the first part of the proof.   We next claim that 
$X'$ does not have a Euclidean factor. Otherwise, the projection to a line factor of $X'$  is an affine function $f$.
Since $X$ does not have  Euclidean factors,  the restriction of $f$ to $X$ is constant, due to Proposition 4.2.
Hence $X$ is contained in a non-Euclidean  factor of $X'$ and it follows that $X'$ contains $X\times \mathbf R$, in contradiction to \lref{app:nonprod}.
Therefore, $X'$ has a unique decomposition with irreducible non-Euclidean factors.

Let now $X=X_1 \times X_2$ be a product decomposition. We have seen that $X_1$ and $X_2$ are both unbounded. Thus $\bd X_1^\omega$ and $\bd X_2^\omega$ are both non-empty, and of radius~$>\pi/2$ since otherwise $\bd X^\omega = \bd (X_1^\omega \times X_2^\omega)$ would have radius~$\leq \pi/2$, contradicting Lemma~\ref{app:strange}. 

By Lemma~\ref{app:lem:BdMin}, we can find a boundary-minimal subspace $X'_i$ in $X_i^\omega$ with $\bd X'_i = \bd X_i^\omega$. Then $X_1' \times  X_2 '$ has $\partial X^{\omega}$ as its boundary, hence $X_1 ' \times X_2 '$  contains $X'$.  However, the intersection
of $X'$ with any $X_i ' $-fiber of the decomposition $X_1 '\times X_2 '$ is either empty or the whole $X_i '$-fiber by boundary-minimality. We deduce $X'=  X_1 '\times X_2 '$.

  Let now $X=Y \times \bar Y =Z\times \bar Z$  be two decompositions.  Consider the corresponding decompositions 
$X'= Y' \times  \bar Y ' = Z' \times \bar Z'$ constructed as above.  By the canonicity of the decomposition of $X'$, we infer that  for $x \in X$, the factor $Y'$ is a product of $Y' _x \cap Z' _x$ and $Y'_x  \cap \bar Z' _x$, where the subscript $x$ is used to denote  the corresponding fiber through $x$.
Moreover, the projection $Y' _x \to Y' _x \cap Z'_x$  coincides with the nearest point projection $Y'_x \to Z' _x$ 
(and the same for the other factor).  Hence the image of $Y_x$ under this projections is $Z_x \cap Y_x$ and $\bar Z_ x \cap Y_x$ respectively.
This implies that $Y_x$ splits as a product of $Y_x \cap Z_x $ and $Y_x \cap \bar Z_x$. 
The canonicity   of the decomposition of $X$ follows.
\end{proof}

\subsection*{Final adjustments}

The proof of Corollary 1.9 on superrigidity remains valid, since the weaker notion of reduced actions established in Proposition~\ref{app:1.8}(iii) is sufficient to apply Monod's theorem. Corollary 1.10 asserting that the amenable radical $R$ acts trivially on each non Euclidean factor of $X$ remains valid; indeed, one only needs to discuss the action of $R$  on the `bad' factor $T$ from  \lref{app:lem:1.8(iv)}.  However, since $T$ has an empty boundary, it follows from Theorem~1.6 that  $R$ must stabilize a Euclidean subspace of $T$, which must be a point since $\bd T$ is empty. Since the $G$-action is minimal and $R$ is normal, the $R$-action on $T$ must be trivial, as desired.

\end{document}